\documentclass[11pt]{article}
\usepackage[pdftex]{graphicx}
\usepackage[cmex10]{amsmath}
\usepackage{amsmath, amssymb, amsfonts, theorem}
\usepackage{float}
\usepackage{enumerate}
\usepackage{setspace}

\newtheorem{thm}{Theorem}[section]
\newtheorem{lemma}[thm]{Lemma}
\newtheorem{cor}[thm]{Corollary}
\newtheorem{propn}[thm]{Proposition}
\newtheorem{definition}[thm]{Definition}

\theoremstyle{plain}{\theorembodyfont{\rmfamily}%

\theoremstyle{plain}{\theorembodyfont{\rmfamily}%

\theoremstyle{plain}{\theorembodyfont{\rmfamily}%

\theoremstyle{plain}{\theorembodyfont{\rmfamily}%
\newtheorem{remark}[thm]{Remark}

\renewcommand{\equiv}{:=}

\def\enorm#1{\|#1\|_2}

\def\norm#1{\|#1\|} 
\def\range{\mbox{range}}

\def\point {\bar x}
\def\set {\Omega}
\def\N{I\!\!N}
\def\R{{I\!\!R}}
\def\E{\R^n}
\def\Ball{I\!\!B}
\def\A{A}
\def\B{B}
\def\Id{\mbox{Id}}
\def\Fix{{\mbox{Fix\ }}}
\def\argmin#1{\underset{x\in #1}{\text{argmin}}}
\def\supp{\textup{supp}}
\def\rig{\right\rbrace}
\def\lef{\left\lbrace}
\def\rign{\right\|}
\def\lefn{\left\|}
\def\sparse{A_s}

\def\M{M^\dagger}

\newcommand{\dist}[2]{d_{#2}\left( #1\right)}

\begin{document}

\title{Alternating Projections and Douglas-Rachford for Sparse Affine Feasibility\footnote{Submitted to IEEE Transactions on Signal Processing, owner of Copyright.}}

\author{ 
Robert Hesse,\thanks{Robert Hesse is with the Institut f\"ur Numerische und Angewandte Mathematik\
Universit\"at G\"ottingen,\ Lotzestr.~16--18, 37083 G\"ottingen, Germany. E-mail: {\tt hesse@math.uni-goettingen.de}. This author
was supported by DFG grant SFB 755-C2.} \and
D.\ Russell Luke\thanks{D. Russell Luke is with Institut f\"ur Numerische und Angewandte Mathematik\
Universit\"at G\"ottingen,\ Lotzestr.~16--18, 37083 G\"ottingen, Germany. E-mail: {\tt r.luke@math.uni-goettingen.de}.  This author
was supported by DFG grants SFB 755-C2 and SFB 755-A4.} \and
and Patrick Neumann\thanks{Patrick Neumann is with the Institut f\"ur Numerische und Angewandte Mathematik\
Universit\"at G\"ottingen,\ Lotzestr.~16--18, 37083 G\"ottingen, Germany. E-mail: {\tt p.neumann@math.uni-goettingen.de}.
This author was supported by DFG grant GRK1023.} 
}

\maketitle

\begin{abstract}
The problem of finding a vector with the fewest nonzero elements 
that satisfies an underdetermined system of linear equations is 
an NP-complete problem that is typically solved numerically via 
convex heuristics or nicely-behaved nonconvex relaxations.  In   
this work we consider elementary methods based on projections 
for solving a {\em sparse feasibility} problem  
without employing convex heuristics.  In a recent paper Bauschke, 
Luke, Phan and Wang (2014) showed that, locally, the fundamental 
method of alternating projections {\em must} converge linearly to 
a solution to the sparse feasibility problem with an affine constraint. 
In this paper we apply different analytical tools that allow us to 
show {\em global} linear convergence of  alternating projections  
under familiar constraint qualifications.
These analytical tools can also be applied to other algorithms.  This is 
demonstrated with the prominent Douglas-Rachford algorithm where 
we establish local linear convergence of this method applied 
to the sparse affine feasibility problem.  
\end{abstract}



%



\maketitle

\section{Introduction}
%
%
%
%

Numerical algorithms for nonconvex optimization models are often eschewed because the usual optimality criteria around which 
numerical algorithms are designed do not distinguish 
solutions from critical points.   This issue comes into sharp relief with what has become known as 
the {\em sparsity optimization problem} \cite[Eq.(1.3)]{CandTao}:
\begin{equation}\label{eq:sparseaffine}
\textup{minimize }\norm{x}_0~\textup{subject to }Mx=p,
\end{equation}
where $m,n\in\N$, the nonnegative integers, with $m<n$, $M\in\R^{m\times n}$ is a real $m-$by$-n$ matrix of full rank and 
$\norm{x}_0:=\sum_{j=1}^n|\mbox{sign}(x_j)|$ with $\mbox{sign}(0)=0$ is the number of nonzero entries of a real 
vector $x\in \R^n$ of dimension $n$.  The first-order necessary optimality condition for this problem is (formally)
\begin{equation}\label{e:foc dumb}
   0\in \partial\left(\norm{x}_0 + \iota_B(x)\right),
\end{equation}
where $\partial$ is the {\em subdifferential}, 
\begin{equation}\label{e:B}
   B:= \left\{ x\in\R^n\middle|~Mx=p\right\}
\end{equation}
 and $\iota_B(x)=0$ if $x\in B$ and $+\infty$ otherwise. The function $\|\cdot\|_0$ is 
subdifferentially regular \cite{Le12}, so all of the varieties of the subdifferential 
in \eqref{e:foc dumb} are equivalent.  
It can be shown \cite{Hiriart-Urruty12} that 
{\em every} point in $B$ satisfies \eqref{e:foc dumb} and so this  
is uninformative as a basis for numerical algorithms.  

In this note we explore the following question: when do elementary numerical algorithms 
for solving some related nonconvex problem converge locally and/or globally?

The current trend for solving this problem, sparked by the now famous paper of Cand\`es and Tao \cite{CandTao}, 
is to use convex relaxations.  Convex relaxations have the advantage that 
every point satisfying the necessary optimality criteria is also a solution to the relaxed optimization problem.  
This certainty comes at the cost of imposing difficult-to-verify restrictions on the affine constraints  
\cite{TillmannPfetsch13} in order to guarantee 
the correspondence of solutions to the relaxed problem to solutions to the original problem.   Moreover,  
convex relaxations can lead to a tremendous increase
in the dimensionality of the problem (see for example \cite{CandesEldarStrohmerVononinski}).  

In this work we present a different {\em nonconvex} approach; one with the advantage that the available 
algorithms are simple to apply, (locally) linearly convergent, and the problem formulation stays close
in spirit if not in fact to the original problem, thus avoiding the curse of dimensionality.   We also 
provide conditions under which fundamental algorithms applied to the nonconvex model are 
globally convergent.

Many strategies for relaxing \eqref{eq:sparseaffine} have been studied in the last 
decade.  In addition to convex, and in particular $\ell_1$, relaxations, 
authors have studied dynamically reweighted $\ell_1$ (see \cite{BorLuke10b, CandesBoydWakin07}) as well as 
relaxations to $\ell_p$ semi-metric ($0<p<1$) 
(see, for instance, \cite{LaiWang10}).  
The key to all relaxations, whether they be convex or not, is the correspondence between the 
relaxed problem and \eqref{eq:sparseaffine}.  Cand\`es and Tao \cite{CandTao} introduced the 
restricted isometry property of the matrix $M$ as a sufficient condition for the correspondence of 
solutions to \eqref{eq:sparseaffine} with solutions to the {\em convex} problem of finding the point $x$ 
in the set $B$ with smallest $\ell_1$-norm.  This condition was generalized in
\cite{BloomensathDavies09, BloomensathDavies10, BeckTeb}
in order to show {\em global} convergence of the simple projected gradient 
method for solving the problem
\begin{equation}\label{e:BeckTeboulle}
\textup{minimize }\tfrac12\norm{Mx-p}_2^2~\textup{subject to }x\in \sparse,   
\end{equation}
where 
\begin{equation}\label{e:A}
\sparse:= \left\{ x\in\R^n\middle|~\norm{x}_0\leq s\right\},
\end{equation}
the set of $s$-sparse vectors for a fixed $s\leq n$. Notable in this model is that the sparsity ``objective'' is in the constraint, and one must specify a priori the 
sparsity of the solution.  Also notable is that the problem \eqref{e:BeckTeboulle} is still nonconvex, although 
one can still obtain global convergence results.  

Inspired by \eqref{e:BeckTeboulle}, and the desire to stay as close to \eqref{eq:sparseaffine} as possible, 
we model the optimization problem as a  {\em feasibility} problem
\begin{equation}\label{eq:feasibility}
 \textup{Find }\bar x\in \sparse\cap B,
\end{equation}
where $\sparse$ and $B$ are given by \eqref{e:A} and \eqref{e:B}, respectively.  
For a well-chosen sparsity parameter $s$, solutions to \eqref{eq:feasibility} exactly correspond
to solutions to \eqref{eq:sparseaffine}.  
Such an approach was also proposed in \cite{CombTrus90} where the authors proved local convergence of
a simple alternating projections algorithm for feasibility with a sparsity set.  
Alternating projections is but one of a huge variety of projection 
algorithms for solving feasibility problems. The goal of this 
paper is to show {\em when} and {\em how fast} fundamental projection algorithms applied to this 
nonconvex problem converge.  Much of this depends on the abstract geometric 
structure of the sets $\sparse$ and $B$; for affine sparse feasibility this is well-defined and surprisingly 
simple.

The set $B$ is an affine subspace and $\sparse$ is a nonconvex set. 
However, the set $\sparse$ is the union of finitely many subspaces, each 
spanned by $s$ vectors from the standard basis for $\E$ \cite{BLPWII}. 
We show in  \eqref{eq:P_A} that one can easily calculate a 
\emph{projection} onto $\sparse$.

For $\Omega\subset\R^n$ closed and nonempty, 
we call the mapping $P_{\Omega}:\R^n\rightrightarrows \Omega$ the \emph{projector} onto $\Omega$ defined by 
\begin{equation}\label{d:projection}
P_{\Omega}(x)\equiv\textup{argmin}_{y\in \Omega} \norm{x-y}.  
\end{equation}
This is in general 
a set-valued mapping, indicated by the notation ``$\rightrightarrows$''~\cite[Chapter 5]{VA}.  
We call a point $\point\in P_{\Omega}(x)$ a \emph{projection}.  It is well known that 
if the set $\Omega$ is closed, nonempty and {\em convex} then the projector is single-valued.  In a reasonable abuse 
of terminology and notation, we will write $P_\Omega(x)$ for the (there is only one) projection onto a convex
set $\Omega$.    
An operator closely related to the projector is the reflector. We call the (possibly set-valued) 
mapping $R_\Omega:\R^n\rightrightarrows \R^n$
 the {\em reflector} across $\Omega$ defined by $R_\Omega(x)\equiv 2 P_\Omega(x) - x$.  
We call a point in $R_\Omega(x)$ a {\em reflection}.  As with projections, when $\Omega$ is convex, 
we will write $R_\Omega(x)$ for the (there is only one) reflection.     
The projection/reflection methods discussed in this work are easy to implement, computationally efficient and 
lie at the foundation of many first-order methods for optimization.  
\begin{definition}[alternating projections]\label{AP}
For two closed sets $\Omega_1,\Omega_2\subset\E$ the mapping 
\begin{equation}\label{eq:AP}
T_{AP}x\equiv P_{\Omega_1} P_{\Omega_2} x
\end{equation}
is called the alternating projections operator.  
The corresponding alternating projections algorithm is given by the iteration
$x^{k+1}\in T_{AP}x^k,$ $k\in\N$
with $x^0$ given. 
\end{definition}
\noindent Other well known algorithms, such as steepest descents for minimizing the average of squared distances between 
sets, can be formulated as instances of the alternating projections  algorithm \cite{Pierra76, Pierra84}.  We show below (Corollary \ref{t:API})
that alternating projections  corresponds to projected gradients for problems with special linear structure.  

\begin{definition}[Douglas-Rachford]\label{AAR}
For two closed sets $\Omega_1,\Omega_2\subset\E$ the mapping 
\begin{equation}
T_{DR}x\equiv\frac 12\left(R_{\Omega_1} R_{\Omega_2} x+x\right)\label{eq:AAR}
\end{equation}
is called the Douglas-Rachford operator.
The corresponding  Douglas-Rachford algorithm is the fixed point iteration 
$x^{k+1}\in T_{DR}x^k,$ $k\in \N$
with $x^0$ given.
\end{definition}
\noindent The Douglas-Rachford algorithm \cite{LionsMercier} owes its prominence in large part to its relation via duality to the 
alternating directions method of multipliers (ADMM) for solving constrained optimization problems \cite{Gabay83}.

We present four main results, three of which are new.  The first of these results, Theorem \ref{t:AP convergence}, 
concerns local linear convergence of alternating projections  to a solution of \eqref{eq:feasibility}.  This has been shown, with optimal rates, 
in \cite{BLPWII}.  Our proof uses fundamentally different tools developed in \cite{HesseLuke12}.  It is exactly these
newer tools that enable us to prove the second of our main results, Theorem \ref{t:convergence DR}, 
namely local linear convergence of the Douglas-Rachford algorithm.  Convergence of Douglas-Rachford, 
with rates, for sparse affine feasibility is a new result.  In the remaining two main new results, Corollary \ref{t:API} and  
Theorem \ref{t:APII}, we specify classes of affine subspaces $B$ for which alternating projections is 
{\em globally linearly convergent}. 
This shows that nonconvex models, in this case, can be a reasonable alternative to convex relaxations.

The outline of this paper is as follows.
First we recall some definitions and results from variational analysis regarding alternating projections  and Douglas-Rachford in Section 2.  
We also show in this section local linear convergence of alternating projections.    
In Section 3 we provide conditions on matrices $M$ that guarantee global linear convergence of alternating projections.  
In the same section we formulate different conditions on the matrices $M$ that guarantee global linear convergence of the same
algorithm.  
In Section 4 we show that for most problems of interest in sparse optimization there exist fixed points of Douglas-Rachford that are 
{\em not} in the intersection $\sparse\cap B$.   
On the other hand, we show that locally the iterates of Douglas-Rachford converge with linear rate to a fixed point whose {\em shadow} 
is a solution to \eqref{eq:feasibility}.    Finally in Section 5  we present numerical and analytical examples to illustrate the theoretical results.

\section{Preliminary Definitions and Results}
We use the following notation, most of which is standard.  We  denote the {\em closed} ball of radius $\delta$ 
centered on $\point$ by $\Ball_\delta(\point)$. 
We assume throughout that the matrix $M$ is full rank in the definition of the affine subspace $B$ \eqref{e:B}. 
The nullspace of $M$ is denoted $\ker M$ and  $M^\dagger$ indicates the \emph{Moore-Penrose inverse}, 
defined by
\begin{equation}\label{MoorePenrose}
M^\dagger\equiv M^\top\left(MM^\top\right)^{-1}.
\end{equation}

The inner product of 
two points $x,y\in \E$ is denoted $\langle x,~ y\rangle$.  The orthogonal 
complement to a nonempty affine set $\Omega$ is given by 
\[
  \Omega^\perp\equiv \left\{p\in \E~\left|~ \langle p,~v-w\rangle=0~\forall~v,w\in  \Omega  \right.\right\}. 
\] 
For two arbitrary sets $\Omega_1,\Omega_2\subset\E$ we denote the Minkowski sum by 
$\Omega_1+\Omega_2\equiv \{x_1+x_2~|~ x_1\in \Omega_1\ \text{and}\  x_2\in \Omega_2 \}$.  
The set of fixed points of a self-mapping $T$ is given by $\textup{Fix}~T$.  The identity mapping 
is denoted by $\Id$. 
For a set $ \Omega\subset\R^n$ we define the {\em distance} of a point $x\in \R^n$ to $\Omega$ by 
$d_{ \Omega}(x):=\inf_{y\in  \Omega}\norm{x-y}$.  When $\Omega$ is closed the distance 
is attained at a projection onto $\Omega$, that is, $\dist{x}{\Omega} = \|\point-x\|$ for $\point\in P_\Omega$.  
\subsection{Tools and notions of regularity}
Our proofs make use of some standard tools and notation from variational analysis which we briefly define here. 
We remind the reader of the definition of the projection onto a closed set \eqref{d:projection}. The following 
definition follows \cite[Definition 2.1]{BLPWIa} and is based on  \cite[Definition 1.1 and Theorem 1.6]{Mor06}.
\begin{definition}[normal cones]\label{d:normal cone}
The 
{\em proximal normal cone} $N^P_\set(\point)$ to a closed nonemtpy set $\set\subset\E$ at a point $\point\in\set$ is
defined by 
\[
N^P_\set(\point) := \textup{cone}(P^{-1}_\set(\point)-\point). \label{eq:pnormal}   
\]
The \emph{limiting normal cone}, or simply the {\em normal cone}  
$ N_\set(\point)$ is defined as the set of all vectors that can be written as the limit 
of proximal normals; that is, $\overline v\in N_\set(\point)$ if and only if there exist sequences $(x^k)_{k\in\mathbf{N}}$ in $\set$ and 
$(v^k)_{k\in \mathbf{N}}$ in $N^P_\set(x^k)$ such that $x^k\to \point$ and $v^k\to \overline v$. 
\end{definition}
\noindent The normal cone describes the local geometry of a set. What is meant by {\em regularity} of sets is made precise below. 
\begin{definition}[{$(\varepsilon,\delta)$-subregularity}]\label{epsilondeltasubregular}
$~$
A nonempty set $\set\subset\E$ is \emph{($\varepsilon,\delta$)-sub\-regular} at $\point$ with respect to $U\subset\E$, 
if there exist $\varepsilon\geq 0$ and $\delta>0$ such that
\[
\langle v,z-y\rangle\leq\varepsilon\norm{v}\norm{z-y}   
\]
holds for all $y\in   \set\cap\Ball_\delta(\point),$ $z\in U\cap   \Ball_\delta(\point),$ $v\in  N_\set(y)$. 
We simply say $\set$ is ($\varepsilon,\delta$)-subregular at $\point$ if $U=\{\point\}$.
\end{definition}

\noindent The definition of ($\varepsilon,\delta$)-subregularity was introduced in \cite{HesseLuke12} and is a 
generalization of the notion of ($\varepsilon,\delta$)-regularity introduced in \cite[Definition 8.1]{BLPWIa}.
During the preparation of this article it was brought to our attention that a similar condition appears 
in the context of regularized inverse problems \cite[Corollary 3.6]{JinLorenz}.

We define next some notions of regularity of collections of sets that, together with  ($\varepsilon,\delta$)-subregularity, 
provide sufficient conditions for linear convergence of both 
alternating projections  and Douglas-Rachford.   In the case of Douglas-Rachford, as we shall see, these conditions are also necessary.
Linear regularity, defined next, can be found in \cite[Definition 3.13]{BauBor93}.  Local versions 
of this have appeared under various names in \cite[Proposition 4]{Ioffe2000}, 
\cite[Section 3]{NgaiThera01}, and  \cite[Equation (15)]{Kruger2006}.

\begin{definition}[linear regularity]\label{localregular}~\\
 A collection of closed, nonempty sets $\left(\set_1,\set_2,\dots,\set_m\right)\subset \E$ is called \emph{locally linearly regular} at 
$\point\in\cap_{j=1}^m \set_j$ on $\Ball_\delta(\point)$ if there exists a 
$\kappa>0$ and a $\delta>0$ such that
\begin{equation}\label{eq:locallinear}
d_{\cap_{j=1}^m \set_j}(x)\leq\kappa \max_{i=1,\dots,m}d_{\set_i}(x),\quad\forall x\in  \Ball_\delta(\point). 
\end{equation}
If (\ref{eq:locallinear}) holds at $\point$ for every $\delta>0$ the collection of sets is said to be \emph{linearly regular} there.
The infimum over all $\kappa$ such that (\ref{eq:locallinear}) holds is called \emph{modulus of regularity on  $\Ball_\delta(\point)$}.  If 
the collection is linearly regular one just speaks of the {\em modulus of regularity} (without mention of $\Ball_\delta(\point)$).   
\end{definition}

\noindent There is yet a stronger notion of regularity of collections of sets that we make use of called the 
{\em basic qualification condition for sets} in \cite[Definition 3.2]{Mor06}.  For the purposes of this paper 
we refer to this as {\em strong regularity}.  
\begin{definition}[strong regularity]\label{d:strong reg}
The collection $(\Omega_1,\Omega_2)$ is \emph{strongly regular} at $\point$ if
\begin{equation}\label{eq:strongregularity}
 N_{\Omega_1}(\point)\cap -N_{\Omega_2} (\point)=\{0\}.
\end{equation}
\end{definition}

It can be shown that strong regularity implies local linear regularity (see, for instance \cite{HesseLuke12}).  
Any collection of finite dimensional affine subspaces with nonempty intersection is linearly regular 
(see for instance \cite[Proposition 5.9 and Remark 5.10]{BauBor96}).  Moreover, it is easy to see that, 
 if $\Omega_1$ and $\Omega_2$ are affine subspaces,
\begin{equation}\label{eq:AB affine}
\left(\Omega_1,\Omega_2\right)\mbox{ is strongly regular at any }\point\in\Omega_1\cap\Omega_2 
\iff   \Omega_1^\perp\cap \Omega_2^\perp = \{0\} \quad\mbox{and}\quad  \Omega_1\cap \Omega_2\neq \emptyset.
\end{equation}
In the case where $\Omega_1$ and $\Omega_2$ are affine subspaces we say that the collection is 
strongly regular without mention of any particular point in the intersection  - as long as this is nonempty - since the collection 
is strongly regular at all points in the intersection.
%
%

\subsection{General local linear convergence results}
The algorithms that we consider here are fixed-point algorithms built upon projections onto 
sets.   
%
Using tools developed in \cite{BLPWIb} and \cite{BLPWIa}, alternating projections  applied to \eqref{eq:feasibility} was 
shown in \cite{BLPWII} to be locally linearly convergent
with optimal rates in terms of the {\em Friedrichs angle} between ${\sparse}$ and $B$, 
and an estimate of the radius of convergence.   
Our approach, based on \cite{HesseLuke12},  is in line with \cite{Luke08} 
but does not rely on local firm nonexpansiveness of the fixed point mapping.  
It has the advantage of being 
general enough to be applied to {\em any} fixed point mapping, but the price one pays for 
this generality is in the rate estimates, which may not be optimal or easy to compute.  
We do not present the results of \cite{HesseLuke12} in their full generality, but focus instead on the 
essential elements for affine feasibility with sparsity constraints. 
\begin{lemma}[local linear convergence of alternating projections ]\label{convergence:AP}
(See \cite[Corollary 3.13]{HesseLuke12}.)
Let\\ the collection $(\Omega_1, \Omega_2)$ be \emph{locally linearly regular} 
at $\point\in \Omega := \Omega_1\cap \Omega_2 $  with modulus of regularity $\kappa$ on $\Ball_\delta(\point)$
and let $ \Omega_1$ and $ \Omega_2 $ be 
$(\varepsilon,\delta)-$subregular at $\point$.  
For any $x^0\in \Ball_{\delta/2}(\point)$, generate the sequence $\left(x^k\right)_{k\in\N}\subset \R^n$ by alternating projections, 
that is, $x^{k+1}\in T_{AP}x^k$.
Then
\[
 d_\Omega({x^{k+1}}) \leq \left(1-\frac 1{\kappa^2}+\varepsilon\right)d_\Omega\left( x^k\right).
\]
\end{lemma}

\noindent 
In the analogous statement for the Douglas-Rachford algorithm, we defer, for the sake of simplicity, 
characterization of the constant in the asserted linear convergence rate.  A more refined analysis of 
such rate constants and their geometric interpretation is the subject of future research.   
\begin{lemma}[local linear convergence of Douglas-Rachford]\label{convergence:AAR}(See \cite[Corollary 3.20]{HesseLuke12}.)
 Let  $\Omega_1,\Omega_2$ be two affine subspaces with $\Omega_1\cap \Omega_2\neq \emptyset$. 
The Douglas-Rachford algorithm converges to $\Omega_1\cap \Omega_2$ for all $x^0\in \E$  
if and only if the collection $(\Omega_1, \Omega_2)$ is strongly regular, in which case, convergence is linear.
\end{lemma}






\section{Sparse Feasibility with an Affine Constraint: local and global convergence of alternating projections}
We are now ready to apply the above general results to affine sparse feasibility.  We begin with 
characterization of the regularity of the sets involved.

\subsection{Regularity of sparse sets}
We specialize to the case where $\B$ is an affine subspace defined by \eqref{e:B} 
and $\sparse$ defined by \eqref{e:A} is the set of vectors with at most $s$ nonzero elements.
Following \cite{BLPWII} we decompose the set $\sparse$ into a union of subspaces. 
For $a\in\R^n$ define the \emph{sparsity subspace} associated with $a$ by
\begin{equation}\label{e:supp}
    \supp(a):=\left\{ x\in\R^n\middle|~x_j=0~\textup{if } a_j=0\right\},
\end{equation}
and the mapping
\begin{equation}\label{eq:integermapping}
I:\R^n\to\{1,\dots,n\},\quad
x\mapsto\left\{i\in\{1,\dots,n\}\middle|~x_i\neq0\right\}.
\end{equation}
Define
$ \mathcal J:=2^{\left\{1,2,\dots,n\right\}}~ \textup{and }~ \mathcal J_s:=\left\{ J\in \mathcal J\middle |~ J \mbox { has } s\mbox{ elements}\right\}$.
The set $\sparse$ can be written as the union of all subspaces indexed by $J\in \mathcal J_s$ \cite[Equation (27d)]{BLPWII}, 
\begin{equation}\label{sparsedecomp}
 \sparse=\bigcup_{J\in \mathcal J_s} A_J,
\end{equation}
where $A_J:=\textup{span}\left\{e_i\middle|~i\in J\right\}$ and $e_i$ is the $i-$th standard unit vector in $\R^n$. 
For $x\in\R^n$ we define the set of $s$ largest coordinates in absolute value
\begin{equation}\label{e:Cs}
C_s(x):=\left\{J\in \mathcal J_s \middle|~ \min_{i\in J}|x_i|\geq \max_{i\notin J} |x_i|\right\}.
\end{equation}
The next elementary result will be useful later. 
\begin{lemma}\label{t:dist AJ}(See \cite[Lemma 3.4]{BLPWII})~\\
Let $a\in \sparse$ and assume $s\leq n-1$. Then
\begin{equation}\label{eq:distance2A}
\min \left\{ d_{A_J}(a)~\middle|~ a\notin A_J,~J\in \mathcal J_s\right\}=\min \left\{|a_j|~\middle|~j\in I(a)\right\}.
\end{equation}
\end{lemma}

Using the above notation, the normal cone to the sparsity set $\sparse$ at $a\in \sparse$  has the following closed-form 
representation (see \cite[Theorem 3.9]{BLPWII} and \cite[Proposition 3.6]{Luke12} for the general matrix representation).
\begin{eqnarray}
N_{\sparse}(a)&=&\left\{ \nu\in\R^n\middle|~\norm{\nu}_0\leq n-s\right\}\cap\left(\textup{supp}(a)\right)^\perp\nonumber\\
&=&\bigcup_{J\in\mathcal J_s, I(a)\subseteq J} A_J^\perp.\label{eq:sparseNcone}
\end{eqnarray}
The normal cone to the affine set $B$ also has a simple closed form, namely $N_B(x)=B^\perp$ (see for example \cite[Proposition 1.5]{Mor06}).  Let $y\in \E$ be a point such that $My = p$. Note
that $\ker M$ is the subspace parallel to $B$, i.e. $\ker M = B+\{-y\}$. 

This notation yields the following explicit representations for the projectors onto ${\sparse}$ 
\cite[Proposition 3.6]{BLPWII} and $B$:
\begin{eqnarray}\label{eq:P_A}
&& P_B x:=x-M^\dagger(Mx-p)\quad \mbox{ and }\quad P_{\sparse} (x)\equiv \bigcup_{J\in C_s(x)} P_{A_J}x,
\end{eqnarray}
 where $M^\dagger$ is given by \eqref{MoorePenrose} and
\begin{eqnarray} (P_{A_J}x)_i=\left\{\begin{array}{cc}
x_i, & i\in J,\\~ 0, & i\notin J    
\end{array}\right. .
\end{eqnarray}

%
We collect next some facts about the projectors and reflectors of $\sparse$ and $B$.  
We remind the reader that, 
in a slight abuse of notation, since the set $B$ is convex, we make no distinction between
the projector $P_B(x)$ and the projection $\point\in P_B(x)$.  
\begin{lemma}\label{t:stay}
Let ${\sparse}$ and $B$ be defined by (\ref{e:A}) and (\ref{e:B}).  Let $a\in \sparse$ and $b\in B$.   
For any  $\delta_a\in(0,\min\left\{ |a_j|~\middle| ~j\in I(a)\right\})$ and 
$\delta_b\in(0,\infty)$
the following hold:
\begin{enumerate}[(i)]
   \item\label{t:stay1} $P_B(x)\in\Ball_{\delta_b}(b)$ for all $x\in \Ball_{\delta_b}(b)$;
\item\label{t:stay2} $P_{\sparse}(x)\subset \Ball_{\delta_a/2}(a)$ for all $x\in \Ball_{\delta_a/2}(a)$;
\item\label{t:stay3} $R_B(x)\in\Ball_{\delta_b}(b)$ for all $x\in \Ball_{\delta_b}(b)$;
\item\label{t:stay4} $R_{\sparse}(x)\subset\Ball_{\delta_a/2}(a)$  for all $x\in \Ball_{\delta_a/2}(a)$.
\end{enumerate}
\end{lemma} 
\begin{proof}
\eqref{t:stay1}.  This follows from the fact that the projector is nonexpansive, since $B$ is convex 
and $\norm{P_Bx-b}=\norm{P_Bx-P_Bb}\leq\norm{x-b}$. (In fact, the projector is firmly nonexpansive
as shown, for example, in \cite[Lemma 1.2]{Zarantonello}.)

\eqref{t:stay2}.  
Let  $x\in \Ball_{\delta_a/2}(a)$. For any $i\in I^\circ(a):=\{i:a_i=0 \}$, we have 
$|x_i-a_i| =  |x_i|\leq \delta_a/2$.  Moreover, for all  
$j\in I(a):=\{j:a_j\neq 0 \}$, we have $|x_j - a_j|\leq \delta_a/2$ and so
$|x_j| >\delta_a/2$ for all $j\in I(a)$.  
Altogether this means that $|x_j|> |x_i|$ for all $i\in I^\circ(a), j\in I(a)$.
Therefore the indices of the nonzero elements of $a$
correspond exactly to the indices of the $|I(a)|$-largest elements of $x$, where $|I(a)|$
denotes the cardinality of the set $I(a)$.  Since $|I(a)|\leq s$, the projector of $x$ need not 
be single-valued. (Consider the case $a=(1,0,\dots,0)$ and $x=(1, \delta/4,\delta/4, 0,\dots,0)$ and $s=2$.)  
Nevertheless, for all $x^+\in P_{\sparse}(x)$ we have $a\in\supp(x^+)$ where $\supp(x^+)$ is defined by \eqref{e:supp}. 
Since $\supp(x^+)$ is a subspace, $x^+$ is the orthogonal projection of $x$ onto a subspace, hence by 
Pythagoras' Theorem
\begin{equation} 
\begin{array}{rrl}
&{\norm{x-x^+}}_2^2+  \norm{ x^+-a}_2^2 &= {\norm{x-a}}_2^2 \\
\mbox{ and } & \norm{ x^+-a}_2\leq {\norm{x-a}}_2 & \leq\frac{\delta}{2}.
\end{array}
\end{equation}
Thus $P_{\sparse} x\subset \Ball_{\delta_a/2}(a)$.

\eqref{t:stay3}.  Since the reflector $R_B$ is with respect to an 
affine subspace containing $b$ a simple geometric 
argument shows that for all $x$ we have $\|R_Bx-b\|=\|x-b\|$.  The result follows immediately. 

\eqref{t:stay4}.  As in the proof of \eqref{t:stay2}, for all $x\in \Ball_{\delta_a/2}$ we have 
$a\in\supp(x^+)$ for each $x^+\in P_{\sparse}(x)$.  In other words, the projector, and hence the 
corresponding reflector, is with respect to a subspace containing $a$.  
Thus, as in \eqref{t:stay3}, $\|R_{\sparse}x - a\|=\|x-a\|$, though in this case only for $x\in\Ball_{\delta_a/2}$. 
\end{proof}

The next lemma shows that around any point $\point\in\sparse$ the set $\sparse$ is the union of 
subspaces in $\sparse$ containing $\point$.  Hence around any point $\point\in\sparse\cap B$ 
the intersection $\sparse\cap B$ can be described locally
as the intersection of subspaces and the affine set $B$, each containing $\point$.   
\begin{lemma}\label{t:ugh}
Let $\point\in\sparse\cap B$ with $0<\|\point\|_0\leq s$.  Then for all $\delta<\min\{|\point_i|: \point_i\neq 0\}$
we have
\begin{equation}\label{e:roof}
\sparse\cap \Ball_{\delta}(\point)=\underset{J\in\mathcal{J}_s,~
                                                I(\point)\subseteq J}{\bigcup}A_J\cap \Ball_{\delta}(\point)
\end{equation}
and hence    
\begin{equation}\label{e:red}
\sparse\cap B\cap \Ball_{\delta}(\point)=\underset{J\in\mathcal{J}_s,~
                                                I(\point)\subseteq J}{\bigcup}A_J\cap B\cap \Ball_{\delta}(\point)
\end{equation}
If in fact $\|\point\|_0=s$, then there is a unique
$J\in \mathcal{J}_s$ such that for all $\delta<\min\{|\point_i|: \point_i\neq 0\}$
we have $\sparse\cap \Ball_{\delta}(\point)=A_J\cap \Ball_{\delta}(\point)$ and hence 
$\sparse\cap B\cap\Ball_{\delta}(\point)=A_J\cap B\cap\Ball_{\delta}(\point)$.  
\end{lemma}
\begin{proof}
If $s=n$, then the set $\sparse$ is all of $\R^n$ and both statements are trivial.  
For the case $s\leq n-1$,  
choose any  $x\in \Ball_\delta(\point)\cap \sparse$.  From the definition of $\delta$ and Lemma \ref{t:dist AJ}
we have that, for any  $J\in \mathcal J_s$, if $\point\notin A_J$ then $x\notin A_J$.  By contraposition, therefore, 
$x\in A_J$ implies that $\point \in A_J$, hence, for each $x\in   \Ball_\delta(\point)\cap \sparse$, we have 
$x\in   \Ball_\delta(\point)\cap A_{I(x)}$ where 
$I(\point)\subseteq I(x)\in \mathcal{J}_s$.  
The intersection $ \Ball_\delta(\point)\cap \sparse$ is then the 
union over all such intersections as given by \eqref{e:roof}.  Equation \eqref{e:red} is an immediate consequence of 
\eqref{e:roof}. 

If, in addition $\|\point\|_0=s$, then the cardinality of $I(\point)$ is $s$ and by \cite[Lemma 3.5]{BLPWII} 
$C_s(\point) = \{I(\point)\}$, 
where  $C_s(\point)$ is given by \eqref{e:Cs}.
This means that if $\point$ has sparsity $s$, then there is exactly one subspace $A_J$ with index set $J := I(\point)$ in $\mathcal{J}_s$ 
containing $\point$.
By Lemma \ref{t:dist AJ}, $\dist{\point}{\sparse\setminus A_J}= \min \left\{|\point_j|~\middle|~j\in J\right\} >\delta$.
From this we conclude the equality  $\sparse\cap\Ball_{\delta}(\point)=A_J\cap\Ball_{\delta}(\point)$ and hence 
$\sparse\cap B\cap\Ball_{\delta}(\point)=A_J\cap B\cap\Ball_{\delta}(\point)$, as claimed.
\end{proof}

We conclude this introductory section with a characterization of the sparsity set $\sparse$. 
\begin{thm}[regularity of $\sparse$]\label{t:Asubregular}
 At any point $\point\in \sparse\backslash\{0\}$ the set $\sparse$ is $(0,\delta)$-subregular at $\point$ for 
$\delta\in (0,\min\left\{ |\point_j|~\middle| ~j\in I(\point)\right\}).$
On the other hand, the set $\sparse$ is not $(0,\delta)$-subregular at $\point\in \sparse\backslash\{0\}$ for any 
$\delta\geq\min\left\{ |\point_j|~\middle| ~j\in I(\point)\right\})$. 
In contrast, at $0$ the set $\sparse$ is $(0,\infty)$-subregular.
 \end{thm}
\begin{proof}
Choose any $x\in \Ball_\delta(\point)\cap \sparse$ and  any $v\in N_{\sparse}(x)$.  
By the characterization of the normal cone in (\ref{eq:sparseNcone}) there is some $J\in \mathcal{J}_s$ with 
$I(x)\subseteq J$ and $v\in A^\perp_J\subset N_{\sparse}(x)$.  As in the proof of Lemma \ref{t:ugh}, for 
any $\delta\in (0,\min\left\{ |\point_j|~\middle| ~j\in I(\point)\right\})$ we have 
$I(\point)\subseteq I(x)$, hence $\point-x~{\in A_J}$ and thus
$\langle {v},{\point-x\rangle}=0$. 
By the definition of $(\varepsilon, \delta)$-regularity (Definition \ref{epsilondeltasubregular}) 
$\sparse$ is $(0,\delta)$-subregular as claimed.  

That $\sparse$ is not $(0,\delta)$-subregular at $\point\in\sparse\setminus\{0\}$ for any 
$\delta\geq\min\left\{ |\point_j|~\middle| ~j\in I(\point)\right\})$ follows from the failure of 
Lemma \ref{t:ugh} on balls larger than $\min\left\{ |\point_j|~\middle| ~j\in I(\point)\right\}$. 
Indeed, suppose \newline $\delta$ $\geq$ $\min\left\{ |\point_j|~\middle| ~j\in I(\point)\right\}$, then 
by Lemma \ref{t:dist AJ} there is a point $x\in \Ball_\delta(\point)\cap \sparse$ with $x\in A_J\subset\sparse$ 
but $\point\notin A_J$.   Now we choose $v\in A^\perp_J\subset N_{\sparse}(x)$.  Since 
$\point\notin A_J$, then $\point-x~{\notin A_J}$ and thus
$|\langle {v},{\point-x\rangle}|>0$.  Since  $N_{\sparse}(x)$ is a union of subspaces, the sign of 
$v$ can be chosen so that $\langle {v},{\point-x\rangle}>0$, in violation of $(0,\delta)$-subregularity.

For the case $\point=0$, by (\ref{eq:sparseNcone}) for any $x\in\sparse$ and 
$v\in N_{\sparse}(x)$ we have $\langle {v},{x}\rangle=0$, since $\textup{supp}(x)^\perp \perp \textup{supp}(x)$, 
which completes the proof.
\end{proof}

\subsection{Regularity of the collection $(\sparse, B)$}
We show in this section that the collection $(\sparse, B)$ is locally linearly regular as long as the intersection is non\-empty. 
We begin with a technical lemma.
\begin{lemma}[linear regularity under unions]\label{t:union}
Let $\left(\Omega_1,\Omega_2,\dots,\Omega_{m}, \Omega_{m+1}\right)$ be a collection of non\-empty subsets of $\E$ with 
nonempty intersection.  Let 
$\point\in \left(\cap_{j=1}^m\Omega_j\right)\cap \Omega_{m+1}$.  Suppose
that, for some $\delta>0$, the pair $\left(\Omega_j, \Omega_{m+1}\right)$ is locally linearly regular 
with modulus $\kappa_j$ on $\Ball_\delta(\point)$ for each $j\in \{1,2,\dots, m\}$.  
Then the collection $\left(\bigcup_{j=1}^m\Omega_j, \Omega_{m+1}\right)$ is locally linearly regular at $\point$ on $\Ball_\delta(\point)$
with modulus $\overline{\kappa}=\max_j\{\kappa_j\}$.  
\end{lemma}
\begin{proof}
Denote $\Gamma\equiv \bigcup_{j=1}^m\Omega_j$.  
First note that for all $x\in\Ball_\delta(\point)$ we have
\begin{equation}\label{e:obvious}
\dist{x}{\Gamma\cap \Omega_{m+1}} = \min_{j}\left\{\dist{x}{\Omega_j\cap \Omega_{m+1}}\right\}
\leq \min_{j}\left\{\kappa_j\max\{\dist{x}{\Omega_j}, ~\dist{x}{\Omega_{m+1}}\}\right\},
\end{equation}
where the inequality on the right follows from the assumption that $\left(\Omega_j, \Omega_{m+1}\right)$ 
is locally linearly regular with modulus $\kappa_j$ on $\Ball_\delta(x)$.
Let $\overline\kappa\geq\max_j\{\kappa_j\}$.  Then   
\begin{equation}\label{e:obvious2}
\dist{x}{\Gamma\cap \Omega_{m+1}} 
\leq \overline\kappa\min_{j}\left\{\max\{\dist{x}{\Omega_j}, ~\dist{x}{\Omega_{m+1}}\}\right\} = \overline\kappa\max\left\{\min_j\{\dist{x}{\Omega_j}\},~\dist{x}{\Omega_{m+1}} \right\}.
\end{equation}
This completes the proof.
\end{proof}

\begin{thm}[regularity of $(\sparse, B)$]\label{t:(A,B) regular}
Let ${\sparse}$ and $B$ be defined by (\ref{e:A}) and (\ref{e:B}) with $\sparse\cap B\neq \emptyset$.  
At any $\point\in {\sparse}\cap B$ and for any $\delta\in(0,\min\left\{ |\point_j|~\middle| ~j\in I(\point)\right\})$ 
the collection $(\sparse, B)$ is locally linearly regular on $\Ball_{\delta/2}(\point)$
 with modulus of regularity 
$\overline{\kappa}=\underset{J\in\mathcal{J}_s, I(\point)\subseteq J}{\max}\{\kappa_J\}$ where $\kappa_J$ is the 
modulus of regularity of the collection $(A_J,B)$.
\end{thm}
\begin{proof}
For any $\point\in\sparse\cap B$ we have $\point\in A_J\cap B$ for all $J\in\mathcal{J}_s$ with 
$I(\point)\subseteq J$ and thus $(A_J, B)$ is linearly regular with modulus of regularity $\kappa_J$ 
\cite[Proposition 5.9 and Remark 5.10]{BauBor96}.  
Define 
\[
\overline{\sparse}\equiv \underset{J\in\mathcal{J}_s,~ I(\point)\subseteq J}{\bigcup}A_J.
\]

Then by Lemma \ref{t:union} the 
collection 
$
\left(  \overline{\sparse}, B\right)
$
is linearly regular at $\point$ with modulus of regularity 
$\overline{\kappa}\equiv \underset{J\in\mathcal{J}_s,~ I(\point)\subseteq J}{\max}\{\kappa_J\}$.  By 
Lemma \ref{t:ugh}  $\sparse\cap \Ball_{\delta/2}(\point)=\overline{\sparse}\cap\Ball_{\delta/2}(\point)$ 
for any $\delta\in(0,\min\left\{ |\point_j|~\middle| ~j\in I(\point)\right\})$.  
Moreover, by Lemma \ref{t:stay}\eqref{t:stay2}, for all $x\in \Ball_{\delta/2}(\point)$, we have 
$P_{\sparse} x \subset \Ball_{\delta/2}(\point)$,  and thus $P_{\sparse} x = P_{\overline{\sparse}} x$. In 
other words, $\dist{x}{\sparse} = \dist{x}{\overline{\sparse}}$ for all $x\in \Ball_{\delta/2}(\point)$,
hence the collection $(\sparse, B)$ is locally linearly regular on $\Ball_\delta(\point)$ with 
modulus $\overline{\kappa}$.  This completes the proof.   
\end{proof}
\begin{remark}
   A simple example shows that the collection $(\sparse, B)$ need not be linearly regular.  Consider the 
sparsity set $A_1$, the affine set $B=\{(1, \tau, 0)~|~\tau\in\R\}$ and the sequence of points 
$(x^k)_{k\in\N}$ defined by $x^k=(0,k,0)$.  Then $A_1\cap B=\{(1, 0, 0)\}$ and 
$\max\{\dist{x^k}{A_1}, \dist{x^k}{B}\}=1$ for all $k$ while $\dist{x^k}{A_1\cap B}\to \infty$ as $k\to \infty$.  
\end{remark}

\subsection{Local linear convergence of alternating projections }
The next result shows the local linear convergence of alternating projections  to a solution of \eqref{eq:feasibility}. 
This was also shown in \cite[Theorem 3.19]{BLPWII} using very different techniques.  The approach taken 
here based on the modulus of regularity $\kappa$ on $\Ball_\delta(x)$ is more general, that is, it can be applied to other 
nonconvex problems, but the relationship between the modulus of regularity and the 
{\em angle of intersection} which is used to characterize the optimal rate of convergence \cite[Theorem 2.11]{BLPWII} 
is not fully understood.

\begin{thm}\label{t:AP convergence}
Let ${\sparse}$ and $B$ be defined by (\ref{e:A}) and (\ref{e:B}) with nonempty intersection and let $\point\in {\sparse}\cap B$.  
Choose $0<\delta<\min\left\{ |\point_j|~\middle| ~j\in I(\point)\right\}$.
For $x^0\in \Ball_{\delta/2}(\point)$ the alternating projections  iterates converge linearly to the intersection ${\sparse}\cap B$ with rate
$\left(1-\frac 1{\kappa^2}\right)$ where $\kappa$ is the modulus of regularity 
of $(\sparse, B)$ on $\Ball_\delta(\point)$ (Definition \ref{localregular}). 
\end{thm}
\begin{proof}
By Lemma \ref{t:stay}\eqref{t:stay1} and \eqref{t:stay2} the projections $P_B$  and $P_{\sparse}$
each map $\Ball_{\delta/2}(\point)$ to itself, hence their composition maps $\Ball_{\delta/2}(\point)$ to itself. 

Finally, we show that we may apply Lemma \ref{convergence:AP}. The set $B$ is $(0,+\infty)$-subregular at every point 
in $B$ (i.e., convex) and by Theorem \ref{t:Asubregular} the sparsity set ${\sparse}$ is $(0,\delta)-$ subregular at $\point$. 
Lastly, by Theorem \ref{t:(A,B) regular} the pair $(\sparse, B)$ is locally linearly regular at $\point$ on 
$\Ball_\delta(\point)$ for any $\delta\in (0, \min\left\{ |\point_j|~\middle| ~j\in I(\point)\right\})$.  The assertion 
then follows from Lemma \ref{convergence:AP} with $\epsilon=0$.
\end{proof}

\begin{remark}\label{r:s too big}
The above result does \emph{not} need an exact a priori assumption on the sparsity $s$.
If there is a solution $\overline{x}\in {\sparse}\cap \B$, then $\|\overline{x}\|_0$ can be smaller than $s$ and, geometrically speaking, 
$\overline{x}$ is on a crossing of linear subspaces contained in ${\sparse}$.  It is also worth noting that the 
assumptions are also {\em not} tantamount to local convexity.   In the case that $B$ is a subspace, the point $0$ is 
trivially a solution to \eqref{eq:feasibility} (and, for that matter \eqref{eq:sparseaffine}).  The set ${\sparse}$ is 
not convex on any neighborhood of $0$, however the assumptions of Theorem \ref{t:AP convergence}
hold, and alternating projections  indeed converges locally linearly to $0$, regardless of the size of the parameter $s$.  
\end{remark}

\subsection{Global convergence of alternating projections}

Following \cite{BeckTeb} where the authors consider problem \eqref{e:BeckTeboulle}, we present a sufficient condition for global linear 
convergence of the alternating projections  algorithm for affine sparse feasibility.  Though our presentation is modeled after \cite{BeckTeb} this work 
is predated by the nearly identical approach developed in \cite{BloomensathDavies09,BloomensathDavies10}.  We also note that 
the arguments presented here do not use any structure that is particular to $\R^n$, hence the results
can be extended, as they were in \cite{BeckTeb}, to the problem of finding the intersection of the set of 
matrices with rank at most $s$ and an affine subspace in the Euclidean space of matrices.  Since this 
generalization complicates the local analysis, we have chosen to limit our scope to $\R^n$.   

Key to the analysis of \cite{BloomensathDavies09, BloomensathDavies10, BeckTeb} are the following well-known 
restrictions on the matrix $M$.  
\begin{definition}\label{d:sRIP}
The mapping $M: \R^n\to \R^m$ satisfies the  \emph{restricted isometry property} of order $s$, if there exists $0\leq \delta \leq 1$ such that
\begin{equation}\label{eq:RIP}
(1-\delta)\enorm{x}^2\leq \enorm{Mx}^2 \leq (1+\delta)\enorm{x}^2 \quad \forall x\in {\sparse}.
\end{equation}
The infimum $\delta_s$ of all such $\delta$ is the \emph{restricted isometry constant}.
\newline
The mapping $M: \R^n\to \R^m$ satisfies the \emph{scaled/asymmetric restricted isometry property} of order $(s, \alpha)$ for $\alpha>1$, if 
there exist $\nu_s, \mu_s >0$ with $1\leq \frac{\mu_s}{\nu_s}< \alpha$ such that
\begin{equation}\label{eq:SRIP}
   \nu_s\enorm{x}^2\leq \enorm{Mx}^2 \leq \mu_s\enorm{x}^2 \quad \forall x\in {\sparse}.
\end{equation}
\end{definition}

\noindent The restricted isometry property \eqref{eq:RIP} was introduced in \cite{CandTao}, while the asymmetric 
version \eqref{eq:SRIP} first appeared in \cite[Theorem 4]{BloomensathDavies09}. 
Clearly \eqref{eq:RIP} implies \eqref{eq:SRIP}, since if a matrix $M$ satisfies \eqref{eq:RIP} of order $s$ with 
restricted isometry constant $\delta_s$, then it also satisfies \eqref{eq:SRIP} of order $\left(s, \beta  \right)$ for 
$\beta >\frac{1+\delta_s}{1-\delta_s}$.

To motivate the projected gradient algorithm given below, note that any solution to \eqref{eq:feasibility} is also a 
solution to 
\begin{equation}\label{eq:P2}
\textup{Find }\point \in S\equiv \textup{argmin}_{x\in{\sparse}} \ \frac{1}{2}\enorm{Mx-p}^2 .
\end{equation}
Conversely, if ${\sparse}\cap B\neq \emptyset$ and $\point$ is in $S$, then $\point$ solves \eqref{eq:feasibility}. 
\begin{definition}[projected gradients]\label{PG}
Given a closed set $A\subset\E$, a continuously differentiable function $f:\E\to \R$ and a positive real number $\tau$, the mapping 
\begin{equation}\label{eq:PG}
T_{PG}(x; \tau)=P_A\left(x- \frac{1}{\tau}\nabla f(x) \right)
\end{equation}
is called the projected gradient operator.  
The projected gradients algorithm is the fixed point iteration 
$$x^{k+1}\in T_{PG}(x^k;\tau_k) = P_A\left(x^k- \frac{1}{\tau_k}\nabla f(x^k) \right),\ k\in \N$$
for $x^0$ given arbitrarily and a sequence of positive real numbers $(\tau_k)_{k\in\N}$. 
\end{definition}
In the context of linear least squares with a sparsity constraint, the projected gradient algorithm is equivalent to what 
is also known as the iterative hard thresholding algorithm (see for instance \cite{BloomensathDavies09, BloomensathDavies10, Cevher12}) 
where the constraint $A=\sparse$ and the projector given by \eqref{eq:P_A} amounts to a thresholding operation on the largest 
elements of the iterate. 

With these definitions we cite a result on  convergence of the projected gradient algorithm applied to 
\eqref{eq:P2} (see  \cite[Theorem 4]{BloomensathDavies10} and 
\cite[Theorem 3 and Corollary 1]{BeckTeb}).
\begin{thm}[global convergence of projected gradients/iterative hard thresholding]\label{thm:BeckTeb}
Let $M$ satisfy \eqref{eq:SRIP} of order $(2s,2)$ and, for any given initial 
point $x^0$, let the sequence $(x^k)_{k\in\N}$ be generated by the projected gradient algorithm 
with 
$A= \sparse$,  $f(x) = \frac{1}{2}\left\| Mx-p\right\|_2^2$  and the constant step size 
$\tau\in[\mu_{2s},2\nu_{2s})$.  Then the iterates converge to the unique global 
solution to \eqref{eq:P2} and  $f(x^k)\to 0$  linearly as $k\to \infty$ with rate  
$\rho=\left(\frac{\tau}{\nu_{2s}}-1\right)<1$, that is, 
\[
f(x^{k+1}) \leq \rho f(x^{k})\qquad (\forall k\in \N).   
\]
\end{thm}
\noindent We specialize this theorem to alternating projections  next. 
\begin{cor}[global convergence of alternating projections I]\label{t:API}
   Let the matrix $M$ satisfy \eqref{eq:SRIP} of order $(2s, 2)$ with  $\mu_{2s}=1$ and $MM^\top = \Id$.  
Then $\sparse\cap\B$ is a singleton and alternating projections applied to \eqref{eq:feasibility} 
converges linearly to $\sparse\cap\B$ with rate $\rho=\left(\frac{1}{\nu_{2s}}-1\right)<1$ 
for every initial point $x^0$.
\end{cor}
\begin{proof}
For $f(x) = \tfrac12\enorm{Mx-p}^2$ we have  $\nabla f(x) = M^\top(Mx -p)$.  The projected gradients  iteration 
with constant step length $\tau=1$ then takes the form 
$$x^{k+1} \in P_{\sparse}\left(x^k - \nabla f(x^k) \right) =  P_{\sparse}\left(x^k - M^\top(Mx^k -p) \right). $$
The projection onto the subspace $B$ is given by (see \eqref{eq:P_A})
$$ P_B x = \left(\Id-M^\top(MM^\top)^{-1}M\right)x + M^\top(MM^\top)^{-1}p.$$
Since $MM^\top =\Id $ this simplifies to $x^k - M^\top(Mx^k-p) = P_B x^k$, hence 
$$x^{k+1} \in P_{\sparse}\left(x^k - \nabla f(x^k) \right) = P_{\sparse}P_B x^k.$$  This shows that  
projected gradients \ref{PG} with unit step length applied to \eqref{eq:P2} with 
$A= \sparse$ and $f(x) = \frac{1}{2}\left\| Mx-p\right\|_2^2$  
is equivalent to the method of alternating projections \ref{AP} applied to \eqref{eq:feasibility}.

To show convergence to a unique solution, we apply Theorem \ref{thm:BeckTeb}, 
for which we must show that 
the step length $\tau=1$ lies in the nonempty interval $[\mu_{2s}, 2\nu_{2s})$.  
By assumption $M$ satisfies \eqref{eq:SRIP} 
of order $(2s,2)$ with $\mu_{2s} = 1$.  Hence $\frac{1}{2}<\nu_{2s}\leq 1$ 
and $\tau = 1$ lies in the nonempty interval $[1, 2\nu_{2s})$.  The 
assumptions of Theorem \ref{thm:BeckTeb} are thus satisfied with $\tau=1$, 
whence global linear convergence to the 
unique solution of \eqref{eq:P2}, and hence \eqref{eq:feasibility}, immediately follows. 
\end{proof}



The restriction to matrices satisfying $MM^\top=\Id$ is very strong indeed.  We consider next  
a different condition that, in principle, can be more broadly applied to the alternating projections algorithm. 
The difference lies in our ansatz: while in \cite{BeckTeb} the 
goal is to minimize  $f(x)\equiv \frac{1}{2}\lefn Mx-p\rign_2^2$ over $x\in\sparse$, we solve instead
\begin{equation}\label{eq:g}
    \underset{x\in\sparse}{\mbox{ minimize }} g(x)\equiv \frac{1}{2}\dist{x}{B}^2.
\end{equation}
These are different objective functions, yet the idea is similar: 
Both functions $f$ and $g$ take the value zero on $\sparse$ if and only if $x \in \sparse\cap B$.
The distance of the point $x$ to $B$, however, is the space of signals, while $f$ measures 
the distance of the image of $x$ under $M$ to the measurement. 
The former is more robust to bad conditioning of the matrix $M\in \R^{m\times n}$ with $m<n$,
since a poorly-conditioned $M$ could still yield a small residual $\frac{1}{2}\lefn Mx-p\rign_2^2$.

Note also that the matrix $\M M$ is the orthogonal projection onto the subspace $\ker(M)^\perp$.
This means that the operator norm of $\M M$ is $1$ and so we have, for all $x\in\R^n$, 
that $\lefn \M Mx\rign_2\leq  \lefn x\rign_2$.  Our second global result for alternating projections
given below, involves a scaled/asymmetric restricted isometry condition analogous to 
\eqref{eq:SRIP} with $M$ replaced by $\M M$.  This only requires 
a {\em lower bound} on the operator norm of  $\M M$ with 
respect to vectors of sparsity $2s$ since the upper bound analogous to \eqref{eq:SRIP} is automatic.  
Specifically, we assume that
\begin{equation}\label{upRIP}
M \mbox{ is full rank and } (1-\delta_{2s}) \lefn x\rign_2^2\leq \lefn \M Mx\rign_2^2 \quad\forall~x\in A_{2s}.
\end{equation}
The condition \eqref{upRIP} can be reformulated in terms of the scaled/asymmetric 
restricted isometry property \eqref{eq:SRIP} and strong regularity of the range of 
$M^\top$ and the complement of each of the subspaces comprising $A_{2s}$.  We remind the 
reader that $A_J\equiv \textup{span}\left\{e_i\middle|~i\in J\right\}$ for 
$J\in \mathcal{J}_{2s}:=\left\{ J\in 2^{\left\{1,2,\dots,n\right\}}\middle |~ J \mbox { has } 2s\mbox{ elements}\right\}$.
\begin{propn}[scaled/asymmetric restricted isometry and strong regularity]\label{t:strong regularity and RIP}
Let $M\in \R^{m\times n}$ with $m\leq n$ be full rank.  Then $M$ satisfies \eqref{upRIP} with $\delta_{2s}\in [0,\frac{\alpha-1}{\alpha})$
for some fixed $s>0$ and $\alpha>1$ if and only if $\M M$ satisfies the
scaled/asymmetric restricted isometry property 
\eqref{eq:SRIP} of order $(2s, \alpha)$  with $\mu_{2s}=1$ and $\nu_{2s}=(1-\delta_{2s})$.  
Moreover, for $M$ satisfying \eqref{upRIP} with $\delta_{2s}\in [0,\frac{\alpha-1}{\alpha})$
for some fixed $s>0$ and $\alpha>1$, for all $J\in \mathcal{J}_{2s}$ the collection $\left(A_J^\perp, \range(M^\top)\right)$ 
is strongly regular (Definition \ref{d:strong reg}), that is,
\begin{equation}\label{e:strong reg A_J-M}
(\forall J\in \mathcal{J}_{2s})\qquad A_J\cap \ker(M) = \{0\}. 
\end{equation}
\end{propn}
\begin{proof}
The first statement follows directly from the definition of the scaled/asymmetric restricted isometry property. 

For the second statement, note that, if $M$ satisfies inequality 
\eqref{upRIP} with $\delta_{2s}\in [0,\frac{\alpha-1}{\alpha})$
for some fixed $s>0$ and $\alpha>1$, then the only element in $A_{2s}$ 
satisfying $\M M x=0$ is $x=0$.  Recall that $\M M$ is the projector onto the space orthogonal to the 
nullspace of $M$, that is, the projector onto the range of $M^\top$.  Thus 
\begin{equation}\label{e:cq}
   A_{2s}\cap [\range(M^\top)]^\perp = \{0\}.
\end{equation}
Here we have 
used the fact that the projection of a point $x$ onto a subspace $\Omega$ is zero if and only if 
$x\in\Omega^\perp$.  Now using the representation for $A_{2s}$ given by \eqref{sparsedecomp}
we have that \eqref{e:cq} is equivalent to 
\begin{equation}\label{e:cq2}
   A_{J}\cap \ker(M^\top) = \{0\}\quad\mbox{for all }J\in\mathcal{J}_{2s}.
\end{equation}
But by \eqref{eq:AB affine} this is equivalent to the strong regularity of  
$\left(A_{J}^\perp, \range(M^\top)\right)$ for all $J\in\mathcal{J}_{2s}.$
\end{proof}

We are now ready to prove one of our main new results.
\begin{thm}[global convergence of alternating projections II]\label{t:APII}
For a fixed $s>0$, let the matrix $\M M$ satisfy \eqref{upRIP}
with $\delta_{2s}\in [0, \frac{1}{2})$ for $M$ in the definition of the affine set 
$B$ given by \eqref{e:B} .
 Then $\ B\cap\sparse$ is a singleton and 
for any initial value $x^0\in\mathbb{R}^n$ the sequence $(x^k)_{k\in\N}$ 
generated by alternating projections (Definition \ref{AP})
converges to $\ B\cap\sparse$ 
with $\dist{x^k}{B}\to 0 $ as $k\to \infty$ at a linear rate with constant 
bounded by $\sqrt{\frac{\delta_{2s}}{1-\delta_{2s}}}$.
\end{thm}

\begin{proof}
From the correspondence 
between \eqref{upRIP} and \eqref{eq:SRIP} in Proposition \ref{t:strong regularity and RIP}, 
we can apply  Theorem \ref{thm:BeckTeb} to the feasibility problem $\mbox{Find } x\in \sparse\cap B^\dagger$,
where $B^\dagger\equiv\{x~|~\M M x = p^\dagger\}$ for $p^\dagger\equiv \M p$.
This establishes that the intersection is a singleton.  
But from \eqref{MoorePenrose} the set $B^\dagger$ is none other than $B$, hence \eqref{upRIP} for $\alpha=2$
implies existence and uniqueness of the intersection $\sparse\cap B$.  

To establish convergence of alternating projections, 
for the iterate $x^k$ define the mapping 
$$q(x,x^k) \equiv g(x^k) + \left\langle x-x^k, \M(Mx^k-p) \right\rangle + \frac{1}{2}\lefn x-x^k\rign_2^2,$$
where $g$ is the objective function defined in \eqref{eq:g}.  
By definition of the projector, the iterate $x^{k+1}$ is a solution to the problem 
$\min \lef q(x,x^k)\ \middle|\ x\in\sparse\rig.$
To see this, recall that, by the definition of the projection,  $g(x^k)=\frac{1}{2}\|x^k-P_B(x^k)\|^2$.
Together with \eqref{eq:P_A} this yields 
\begin{eqnarray}\label{e:turd}
q(x,x^k) &\stackrel{\eqref{eq:P_A}}{=} & \frac{1}{2}\lefn x^k-P_B(x^k)\rign_2^2 + \left\langle x-x^k, x^k-P_B(x^k) \right\rangle  + \frac{1}{2}\lefn x - x^k\rign_2^2  
\nonumber\\ 
&  =  & \frac{1}{2}\lefn x-x^k +  x^k -  P_Bx^k\rign_2^2.
\end{eqnarray}
Now, by definition of the alternating projections sequence, 
$$x^{k+1}\in P_{\sparse}P_B(x^k)=P_{\sparse}\left(x^k - (\Id - P_B)x^k \right),$$ 
which, together with \eqref{e:turd}, yields 
\[
x^{k+1}\in \argmin{\sparse}\left\{\lefn x-\left(x^k - (\Id - P_B)x^k \right)\rign_2^2\right\} = \argmin{\sparse}\{q(x,x^k)\}.   
\]
That is, $x^{k+1}$ is a minimizer of $q(x,x^k)$ in $\sparse$.
On the other hand,  
\begin{align}
 g(x^{k+1}) 
\stackrel{\eqref{eq:P_A}\&\eqref{eq:g}}{=} &~ \frac{1}{2}\lefn \M(Mx^{k+1} - p)\rign_2^2 \nonumber\\
=& ~\frac{1}{2} \lefn \M M(x^{k+1} - x^k) + \M(Mx^k-p)\rign_2^2 \notag \\
 = & ~g(x^k) + \left\langle \M M(x^{k+1}-x^k), \M(Mx^k-p) \right\rangle + \frac{1}{2}\lefn \M M(x^{k+1} - x^k)\rign_2^2 \notag \\
\leq &~ g(x^k) + \left\langle \M M(x^{k+1}-x^k), \M(Mx^k-p) \right\rangle  + \frac{1}{2}\lefn x^{k+1} - x^k\rign_2^2 \notag\\
 \stackrel{\eqref{MoorePenrose}}{=} &~ g(x^k) + \left\langle x^{k+1}-x^k, \M(Mx^k-p) \right\rangle + \frac{1}{2}\lefn x^{k+1} - x^k\rign_2^2 \notag\\
 = &~ q(x^{k+1},x^k), \label{e:decoder ring}
\end{align}
where the inequality in the middle follows from the fact that $\M M$ is an orthogonal projection onto a subspace. 
Hence  $g(x^{k+1}) \leq q(x^{k+1},x^k)$. But since $x^{k+1}$ minimizes $q(x,x^k)$ over $\sparse$,
we know that, for $\{\overline{x}\}= B\cap \sparse$,
\begin{equation}\label{e:pufferfish}
q(x^{k+1},x^k)\leq q(\overline{x},x^k).
\end{equation}
Moreover, by  assumption \eqref{upRIP} we have 
\begin{align}
q(\overline{x},x^k) = & g(x^k) +  \left\langle \overline{x}-x^k, \M(Mx^k-p) \right\rangle + \frac{1}{2}\lefn \overline{x}-x^k\rign_2^2  \nonumber\\
\stackrel{\eqref{upRIP}}{\leq}  & g(x^k) +  \left\langle \overline{x}-x^k, \M(Mx^k-p) \right\rangle + \frac{1}{2\left(1-\delta_{2s}\right)}\lefn \M M(\overline{x}-x^k)\rign_2^2  \nonumber\\
= & g(x^k) +  \left\langle \overline{x}-x^k, \M(Mx^k-p) \right\rangle + \frac{1}{2\left(1-\delta_{2s}\right)}\lefn \M (p-Mx^k)\rign_2^2  \nonumber\\
\stackrel{\eqref{eq:P_A}\&\eqref{eq:g}}{=} & \left(1+\frac{1}{1-\delta_{2s}} \right)g(x^k) +  \left\langle \overline{x}-x^k, \M(Mx^k-p) \right\rangle  \nonumber\\
\stackrel{\eqref{MoorePenrose}}{=} & \left(1 +\frac{1}{1-\delta_{2s}} \right)g(x^k) +  \left\langle \M M(\overline{x}-x^k), \M(Mx^k-p) \right\rangle  \nonumber\\
\stackrel{\eqref{eq:P_A}\&\eqref{eq:g}}{=} & \left(1+\frac{1}{1-\delta_{2s}} \right)g(x^k) - 2 g(x^k) \nonumber\\
= &\frac{\delta_{2s}}{1-\delta_{2s}}g(x^k)\label{e:Voyager}.
\end{align}
When $0\leq\delta_{2s}<\frac{1}{2}$, as assumed, we have 
$0\leq  \frac{\delta_{2s}}{1-\delta_{2s}}<1$.  
Inequalities \eqref{e:decoder ring}-\eqref{e:Voyager} then imply that $\dist{x^k}{B}\to 0$ 
 as $k\to\infty$ at a linear rate for $0\leq\delta_{2s}<\frac{1}{2}$, with constant bounded above by  
$ \sqrt{\frac{\delta_{2s}}{1-\delta_{2s}}}<1.$
Since the iterates $x^{k}$ lie in $\sparse$ this proves convergence 
of the iterates to the intersection $\sparse\cap B$, that is, to $\point$,  as claimed.    
\end{proof}

\section{Sparse Feasibility with an Affine Constraint: local linear convergence of Douglas-Rachford}
We turn our attention now to the Douglas-Rachford algorithm. 
First we present a result that could be discouraging since we show that the Douglas-Rachford operator has a 
set of fixed points that is too large in most interesting cases.
However we show that this set of fixed points has a nice structure guaranteeing 
local linear convergence of the iterates and thus convergence of the shadows to a solution of  
\eqref{eq:feasibility}.  We use the results obtained in \cite{HesseLuke12} in our proofs.
Linear convergence of Douglas-Rachford for the case of $\ell_1$ minimization with an affine constraint 
was obtained by Demanet and Zhang in 
\cite{DemanetZhang13}.  In \cite{BBPW, DemanetZhang13} the authors show that the rate of 
convergence of Douglas-Rachford applied to {\em linear} feasibility problems is the cosine of the Friedrichs 
angle between the subspaces.

\subsection{Fixed point sets of Douglas-Rachford} 
In contrast to  the alternating projections  algorithm, the iterates of the Douglas-Rachford algorithm are not actually the points of interest - it is 
rather the shadows of the iterates that are relevant.  This results in an occasional incongruence between the fixed points
of Douglas-Rachford and the intersection that we seek.  Indeed, this mismatch occurs in the most interesting cases of the affine sparse 
feasibility problem as we show next. 
\begin{thm}\label{t:AAR convergence}
Let ${\sparse}$ and $B$ be defined by (\ref{e:A}) and (\ref{e:B}) and suppose there exists a point $\point\in {\sparse}\cap B$ with 
$\|\point\|_0=s$.  If $s< \textup{rank}(M)$, then 
on all open neighborhoods $\mathcal N$ of $\point\in {\sparse}\cap B$ there exist fixed points $z\in \textup{Fix}~ T_{DR}$ 
with $z\notin {\sparse}\cap B$.
\end{thm}
\begin{proof}
Let $\point\in \sparse\cap B$ with $\|\point\|_0=s$ and set $\delta< \min\{|\point_j| ~|~ \point_j\neq 0\}$.  
By Lemma \ref{t:ugh} we have $\sparse\cap B\cap\Ball_{\delta/2}(\point)=A_J\cap B\cap\Ball_{\delta/2}(\point)$
for a unique $J\equiv I(\point)\in \mathcal{J}_s$.  Thus on 
the neighborhood $\Ball_{\delta/2}(\point)$ the feasibility problems 
$\textup{Find}\,x\in A_J\cap B$, and $\textup{Find}\,x\in \sparse\cap B$
have the same set of solutions.  
We consider the Douglas-Rachford operators applied to these two feasibility problems, for which we  
introduce the following notation:  $T_J\equiv \tfrac12\left(R_{A_J}R_B+\Id\right)$ and 
$T_s\equiv \tfrac12\left(R_{\sparse}R_B+\Id\right)$.  Our proof strategy is to show first that the operators 
$T_J$ and $T_s$ restricted to  $\Ball_{\delta/2}(\point)$ are identical, hence their fixed point sets intersected with 
$\Ball_{\delta/2}(\point)$ are identical.
We then show that under the assumption $s<\textup{rank}\,(M)$ the set $\textup{Fix}\,T_J$ is strictly 
larger than the intersection $A_J\cap B$, hence completing the proof.  

To show that the operators $T_J$ and $T_s$ applied to points $x\in\Ball_{\delta/2}(\point)$ are identical, note that, by 
Lemma \ref{t:stay}\eqref{t:stay2} and 
\eqref{t:stay4}, for all $x\in\Ball_{\delta/2}(\point)$ we have $P_{\sparse}(x)\subset \Ball_{\delta/2}(\point)$ and 
$R_{\sparse}(x)\subset \Ball_{\delta/2}(\point)$.  Moreover by Lemma \ref{t:ugh}, since $\|\point\|_0=s$ we have  
$\sparse\cap \Ball_{\delta}(\point)=A_J\cap \Ball_{\delta}(\point)$.  Thus for all $x\in\Ball_{\delta/2}(\point)$ we have
$P_{\sparse}(x) = P_{A_J}(x)\in \Ball_{\delta/2}(\point)$ and 
$R_{\sparse}(x) = R_{A_J}(x)\in \Ball_{\delta/2}(\point)$.
Also by Lemma \ref{t:stay}, 
$R_Bx\in \Ball_{\delta/2}(\point)$ for $x\in\Ball_{\delta/2}(\point)$.  Altogether, this yields
\begin{equation}\label{eq:Garden}
T_{s}x=\tfrac12\left(R_{\sparse}R_B+\Id\right)x = \tfrac12\left(R_{A_J}R_B+\Id\right)x=T_{J}x \in \Ball_{\delta/2}(\point)  
\end{equation}
for all  $x\in\Ball_{\delta/2}(\point)$.
Hence the operators $T_s$ and $T_J$  and their fixed point sets coincide on $\Ball_{\delta/2}(\point)$.

We derive next an explicit characterization of $\textup{Fix}\,T_J$.  
By \cite[Corollary 3.9]{BauComLuke} and \eqref{eq:AB affine} we have:
\begin{equation}\label{eq:fixedpoints} 
 \begin{array}{rl}
\textup{Fix}\, T_{J} & =(A_J\cap B)+N_{A_J-B}(0)\\
& =(A_J\cap B)+(N_{A_J}(\point)\cap -N_B(\point)) \\
& = (A_J\cap B)+\left(A_J^\perp\cap B^\perp\right).
\end{array}
\end{equation}
The following equivalences show that $A_J^\perp\cap B^\perp$ is nontrivial if $s< \textup{rank}\,(M)$.  Indeed, 

\begin{eqnarray} 
&   \textup{rank}(M) &> s \nonumber\\
\Leftrightarrow &\   \dim(\ker(M)^\perp) &> s \nonumber\\
\Leftrightarrow &\   n-s + \dim(\ker(M)^\perp) &> n \nonumber\\
\Leftrightarrow &  \dim(A_J^\perp) + \dim(\ker(M)^\perp) &> n  \nonumber\\ 
\Leftrightarrow & A_J^\perp\cap B^\perp&\neq \{0\}.\label{eq:Memoryhouse}
\end{eqnarray}
In other words, 
$\textup{Fix}\, T_J $ contains elements from the intersection $A_J\cap B$ and the nontrivial subspace
$A_J^\perp\cap B^\perp$.  This completes the proof.
\end{proof}
\begin{remark}
The inequality \eqref{eq:Memoryhouse} shows that if $\textup{rank}(M) > s$ then the intersection $A_J\cap B$ is {\em not}
{\em strongly regular}, or in other words, if $A_J\cap B$ is strongly regular then $\textup{rank}(M) \leq s$.   
This was also observed in \cite[Remark 3.17]{BLPWII} using tangent cones and 
{\em transversality}.  The simple meaning of these results is that if the sparsity of a feasible point is less 
than the rank of the measurement matrix (the only interesting case in sparse signal recovery) then, since locally
the affine feasibility problem is indistinguishable from simple linear feasibility at points $\point\in\sparse$ with $\|\point\|_0=s$, 
by Lemma \ref{convergence:AAR} the 
Douglas-Rachford algorithm may fail to converge {\em to the intersection} on all balls around a feasible point.  
As we noted in the beginning of this section, however, it is not the fixed points of Douglas-Rachford themselves but rather 
their shadows that are of interest.  This leads to positive convergence results detailed in the next section.
\end{remark}

\noindent 

\subsection{Linear convergence of Douglas-Rachford}
We begin with an auxiliary result that the Douglas-Rachford iteration applied to {\em linear subspaces} converges to its set of fixed points 
with linear rate.  As the sparse feasibility problem reduces locally to finding the intersection of (affine) subspaces, by a translation 
to the origin, results for the case of subspaces will  
yield local linear convergence of Douglas-Rachford to fixed points associated with points $\point\in \sparse\cap B$ such that 
$\|\point\|_0=s$.   Convergence of Douglas-Rachford for convex sets with nonempty intersection was proved first 
by Lions and Mercier \cite{LionsMercier}, but without rate.  (They do, however, achieve linear rates of convergence under strong assumptions
that are not satisfied for convex feasibility.)  As surprising as it may seem, results on the {\em rate} of convergence
of this algorithm even for the simple case of affine subspaces are very recent.  Our proof, based on \cite{HesseLuke12}, 
is one of several independent results (with very different proofs) that we are aware of which have appeared in the 
last several months \cite{BBPW, DemanetZhang13}. 

\subsubsection{The linear case}
The idea of our proof is to show that the set of fixed points of the Douglas-Rachford algorithm applied to the subspaces $A$ and $B$ 
can always be written as the intersection of different subspaces $\widetilde{A}$ and $\widetilde{B}$, 
the collection of which is {\em strongly regular}.  We then show that the iterates of the Douglas-Rachford algorithm applied to the 
subspaces $A$ and $B$ are identical to those of the Douglas-Rachford algorithm applied to the subspaces $\widetilde{A}$ and $\widetilde{B}$.  
Linear convergence of Douglas-Rachford then follows directly from Lemma \ref{convergence:AAR}.

We recall that the set of fixed points of Douglas-Rachford in the case of two linear subspaces $\A\subset\E$ and $\B\subset \E$ is by \cite[Corollary 3.9]{BauComLuke} 
and \eqref{eq:fixedpoints} equal to 
$$\textup{Fix}\, T_{DR} = (\A\cap \B)+\left(\A^\perp\cap \B^\perp\right)$$
for $T_{DR}\equiv\tfrac12\left(R_AR_B+\Id\right)$.  
For two linear subspaces $\A\subset\E$ and $\B\subset \E$  define the enlargements 
$\widetilde{\A}:=\A+ \left(\A^\perp\cap\B^\perp\right)$ and $\widetilde{\B}:= \B+ \left(\A^\perp\cap\B^\perp\right)$.
 By definition of the Minkowski sum these enlargements are given by
\begin{subequations}\label{eq:ABtilde}
\begin{eqnarray} 
\widetilde{\A} &=& \left\lbrace a + n\ \middle|\ a\in\A, n\in \A^\perp\cap\B^\perp \right\rbrace \\   
~\mbox{ and } ~\widetilde{\B} &=& \left\lbrace b + n\ \middle|\ b\in\B, n\in \A^\perp\cap\B^\perp \right\rbrace.
\end{eqnarray}
\end{subequations}
The enlargements $\widetilde{\A}$ and $\widetilde{B}$ are themselves subspaces of $\E$ as the 
Minkowski sum of subspaces. 
\begin{lemma}\label{t:Sarajevo}
The equation
$$C :=\left(\A+ \left(\A^\perp\cap\B^\perp\right)\right)^\perp\cap \left(\B+ \left(\A^\perp\cap\B^\perp\right)\right)^\perp = \{0\}$$
holds for any linear subspaces $\A$ and $\B$ of $\E$, and hence the collection $(\widetilde{A},\widetilde{B})$ is strongly regular for 
any linear subspaces $A$ and $B$.  
\end{lemma}

\begin{proof}
Let $v$ be an element of $C$. 
Because
$C =\widetilde{A}^\perp\cap\widetilde{B}^\perp$,
 we know that 
\begin{equation}
\left\langle v, \widetilde{a}\right\rangle = \left\langle v, \widetilde{b}\right\rangle = 0\quad \textup{for all}\ \widetilde{a}\in\widetilde{A}, \widetilde{b}\in\widetilde{B}.
\end{equation}
Further, since $\A\subset\widetilde{\A}$ and $\B\subset\widetilde{\B}$ we have 
\begin{equation}
\left\langle v, a\right\rangle = \left\langle v, b\right\rangle = 0\quad \textup{for all}\ a\in\A, b\in\B.
\end{equation}
In other words, $v\in\A^\perp$ and $v\in\B^\perp$, so $v\in\A^\perp\cap\B^\perp$.
On the other hand, $\A^\perp\cap\B^\perp\subset\widetilde{A}$ and $\A^\perp\cap\B^\perp\subset\widetilde{\B}$, so we similarly have
\begin{equation}
\left\langle v, n\right\rangle =  0\quad \textup{for all}\ n\in\A^\perp\cap\B^\perp,
\end{equation}
because $A$ and $B$ are subspaces and $v\in C$. 
 Hence  $v$ is also an element of $\left(\A^\perp\cap\B^\perp\right)^\perp$.
We conclude that $v$ can only be zero.
\end{proof}

\begin{lemma}\label{thm:refl}
Let $\A$ and $\B$ be linear subspaces and let $\widetilde{\A}$ and $\widetilde{B}$ be their corresponding enlargements
defined by \eqref{eq:ABtilde}.  
\begin{enumerate}[(i)]
\item\label{thm:refl1} $R_{\A}d=-d$ for all $d\in A^\perp$.
   \item\label{thm:refl2} $R_{\A}x = R_{\widetilde{\A}}x$  for all $x\in\A+\B$.
\item\label{thm:refl3} $R_{\widetilde{\B}}a\in A+B$ for all $a\in A$.
\item\label{thm:refl4} $R_{\widetilde{\A}}R_{\widetilde{\B}}x = R_{\A}R_{\B}x$ for all $x\in\R^n$.
\item\label{thm:refl5} For any $x\in\R^n$ the following equality holds:
$$\frac{1}{2}\left( R_{\widetilde{\A}}R_{\widetilde{\B}}+\Id\right) x = \frac{1}{2}\left( R_{\A}R_{\B}+\Id \right) x.$$
\end{enumerate}
\end{lemma}

\begin{proof}
To prove \eqref{thm:refl1}, let $d\in \A^\perp$ be arbitrary. The 
projection $P_{\A}d$ of $d$ onto $\A$ is the orthogonal projection onto $\A$. The orthogonal projection 
of $d\in \A^\perp$ is the zero vector. This means that $R_{\A}d = (2P_{\A}-\Id)d = -d$.

To show  \eqref{thm:refl2}\footnote{This proof is a simplification of our original proof suggested by an anonymous
referee.} note that $(A^\perp\cap B^\perp) = (A+B)^\perp$ hence $\widetilde A = A + (A+B)^\perp$. 
Now by \cite[Proposition 2.6]{BCL4}, $P_{A + (A+B)^\perp} = P_A + P_{(A+B)^\perp}$.  Hence
for all $x\in\A+\B$, $P_{\widetilde A}x=P_Ax$ and, consequently, $R_{\widetilde A}x=R_Ax$, as claimed.

To prove \eqref{thm:refl3}, let $a\in\A$ and thus $a\in\A+\B$. We note that by \eqref{thm:refl2} with $A$ replaced by $B$
we have $R_{\B}a = R_{\widetilde{\B}}a$. Write $a$ as a sum $b + v$ where $b = P_Ba$ and $v = a-P_Ba$.
We note that $v\in \A+\B$ and so $-v\in \A+\B$. From \eqref{thm:refl1} we conclude, since $\A$ in
\eqref{thm:refl1} can be replaced by $\B$ and $v\in\B^\perp$, that $R_{\B}v = -v$.
Since $b\in\B$, we have $R_{\B}b = 2P_{\B}b-b = b$ and so 
\begin{equation}
R_{\widetilde{\B}}a = R_{\B}a = R_{\B}b + R_{\B}v = b-v\in \A+\B.
\end{equation}

To see \eqref{thm:refl4} let $x\in\R^n$ be arbitrary. 
Define $D:= \A^\perp\cap\B^\perp$. 
Then we can write as $x= a + b + d$ with $a\in A$, $b\in B$ and $d\in D$. 
This expression does not have to be unique since $\A$ and $\B$ may have a nontrivial intersection.
In any case, we have the identity $\langle b,d\rangle = \langle a, d\rangle=0$.
Since $\A$ and $\B$ are linear subspaces, the Douglas-Rachford operator is a linear mapping which together 
with parts \eqref{thm:refl1}-\eqref{thm:refl3} of this lemma yields
\begin{equation} 
\begin{array}{rl}
R_{\A}R_{\B}x & = R_{\A}\left(R_{\B}a + R_{\B}b + R_{\B}d\right) \\
& \stackrel{\eqref{thm:refl1}.}{=} R_{\A}\left(R_{\B}a + b -d\right) \\
& = R_{\A}R_{\B}a + R_{\A}b +R_{\A}(-d)\\
 &\stackrel{\eqref{thm:refl1}.}{=}R_{\A}R_{\B}a + R_{\A}b +d\\
& \stackrel{\eqref{thm:refl2}.}{=} R_{\A}R_{\widetilde{\B}}a + R_{\widetilde{\A}}b +d\\
&  \stackrel{\eqref{thm:refl2}.-\eqref{thm:refl3}.}{=} R_{\widetilde{\A}}R_{\widetilde{\B}}a + R_{\widetilde{\A}}b +d\\
& \stackrel{d\in\widetilde{\A}}{=} R_{\widetilde{\A}}\left(R_{\widetilde{\B}}a + b +d\right) \\
& \stackrel{b,d\in\widetilde{\B}}{=} R_{\widetilde{\A}}\left(R_{\widetilde{\B}}a + R_{\widetilde{\B}}b + R_{\widetilde{\B}}d\right) \\
& = R_{\widetilde{\A}}R_{\widetilde{\B}}x.
\end{array}\label{eq:CraneWife}
\end{equation} 
This proves \eqref{thm:refl4}.

Statement \eqref{thm:refl5} is an immediate consequence of \eqref{thm:refl4}, which completes the proof.
\end{proof}

\begin{propn}\label{prop:conv}
Let $\A$ and $\B$ be linear subspaces and let $\widetilde{\A}$ and $\widetilde{B}$ be their corresponding enlargements 
defined by \eqref{eq:ABtilde}.  The Douglas-Rachford iteration applied to the enlargements
\begin{equation}\label{alg:DRenl}
x^{k+1}= \widetilde{T}_{DR}x^k := \frac{1}{2}\left(R_{\widetilde{\A}}R_{\widetilde{\B}}+\Id\right)x^k
\end{equation}
converges with linear rate to $\textup{Fix } \widetilde{T}_{DR}$ for any starting point $x^0\in\E$.
\end{propn}
\begin{proof}
By Lemma \ref{t:Sarajevo} we know that the only common element in 
\\ $\left(\A+ \left(\A^\perp\cap\B^\perp\right)\right)^\perp$ and $\left(\B+ \left(\A^\perp\cap\B^\perp\right)\right)^\perp$ is the zero vector. 
By Lemma \ref{convergence:AAR} \cite[Corollary 3.20]{HesseLuke12}	 the sequence
$$\widetilde{x}_{k+1} := \frac{1}{2}\left(R_{\widetilde{\A}}R_{\widetilde{\B}}+\Id\right)\widetilde{x}_k$$
converges linearly to the intersection $\widetilde{\A}\cap\widetilde{\B}$ for any starting point $\widetilde{x}_0\in\R^n$.
\end{proof}

Combining these results we obtain the following theorem confirming {\em linear} convergence of the 
Douglas-Rachford algorithm for subspaces.  Convergence of the Douglas-Rachford algorithm for 
{\em strongly regular} affine subspaces was proved in \cite[Corollary 3.20]{HesseLuke12} as
a special case of a more general result \cite[Theorem 3.18]{HesseLuke12} 
about linear convergence of the Douglas-Rachford algorithm for a strongly regular collection of a
{\em super-regular set} \cite[Definition 4.3]{LLM} and an affine subspace.   
Our result below shows that the iterates of the Douglas-Rachford algorithm for 
{\em linearly regular} affine subspaces (not necessarily strongly regular) converge linearly
to the fixed point set.  
An analysis focused only on the affine case in the recent preprint 
\cite{BBPW} also achieves linear convergence of the Douglas-Rachford algorithm.    
\begin{thm}\label{thm:convDR}
For any two affine subspaces $\A, \B\subset \R^n$ with nonempty intersection, 
the Douglas-Rachford iteration
\begin{equation}\label{alg:DR}
x^{k+1}= T_{DR}x^k := \frac{1}{2}\left(R_{\A}R_{\B}+\Id\right)x^k
\end{equation}
converges for any starting point $x^0$ to a point in the fixed point set with linear rate.
Moreover, $P_{\B}\overline{x}\in\A\cap\B$ for $\overline{x}=\lim_{k\to\infty}x^k$.
\end{thm}
\begin{proof}
Without loss of generality, by translation of the sets $A$ and $B$ by $-\point$ for $\point\in A\cap B$, we consider
the case of subspaces.  
By Proposition \ref{prop:conv} Douglas-Rachford applied to the enlargements 
$\widetilde{\A}=\A+ \left(\A^\perp\cap\B^\perp\right)$ and $\widetilde{\B}= \B+ \left(\A^\perp\cap\B^\perp\right)$, 
namely (\ref{alg:DRenl}),  converges  to the intersection $\widetilde{\A}\cap\widetilde{\B}$ with linear rate for any starting point $x^0\in\E$.
By  \cite[Corollary 3.9]{BauComLuke} and \eqref{eq:AB affine}, the set of  fixed points of the Douglas-Rachford algorithm (\ref{alg:DR}) is
\begin{equation}
\textup{Fix}_{T_{DR}} = \left(\A\cap\B\right)+\left(\A^\perp\cap\B^\perp\right) = \widetilde{A}\cap\widetilde{B},
\end{equation}
where the rightmost equality follows from repeated application of the identity 
$(\Omega_1+\Omega_2)^\perp = (\Omega_1^\perp\cap \Omega_2^\perp)$, 
the definition of set addition and closedness of subspaces under addition.  
By Lemma \ref{thm:refl}\eqref{thm:refl5} the iterates of (\ref{alg:DRenl}) are the same as the iterates of (\ref{alg:DR}). 
So the iterates of the Douglas-Rachford algorithm applied to $A$ and $B$  
converge to a point in the set of its fixed points with linear rate.
Finally, by \cite[Corollary 3.9]{BauComLuke}, $ P_{\B}\overline{x}\in \A\cap\B$ for any $\point\in\Fix T_{DR}$.
\end{proof}

\subsubsection{Douglas-Rachford applied to sparse affine feasibility}
We conclude with an application of the analysis for affine subspaces to the case of affine feasibility with 
a sparsity constraint.  
\begin{thm}\label{t:convergence DR}
Let ${\sparse}$ and $B$ be defined by (\ref{e:A}) and (\ref{e:B}) with nonempty intersection and let $\point\in {\sparse}\cap B$ with 
$\norm{\point}_0 = s$.  
Choose $0<\delta<\min\left\{ |\point_j|~\middle| ~j\in I(\point)\right\}$.
For $x^0\in \Ball_{\delta/2}(\point)$ the corresponding Douglas-Rachford iterates converge with linear rate to $\Fix T_{DR}$.  Moreover, 
for any $\hat{x}\in\Fix T_{DR}\cap\Ball_{\delta/2}(\point)$, we have $P_B\hat{x}\in {\sparse}\cap B$. 
\end{thm}
\begin{proof}
By Lemma \ref{t:ugh} we have $\sparse\cap\B\cap \Ball_\delta(\point)=A_J\cap B\cap \Ball_\delta(\point)$
for a unique $J\in \mathcal{J}_s$. 
Thus by \eqref{eq:Garden} at all points in $\Ball_{\delta/2}(\point)$ the Douglas-Rachford operator 
corresponding to $A_s$ and $B$ is equivalent to 
the Douglas-Rachford operator corresponding to $A_J$ and $B$,  whose intersection includes $\point$.
Applying Theorem \ref{thm:convDR}, shifting the subspaces appropriately,  we see that the iterates converge to some point 
$\hat{x}\in\Fix T_{DR}$ with linear rate for all initial points $x^0\in\Ball_{\delta/2}(\point)$.
The last statement follows from \eqref{eq:Garden} and Theorem \ref{thm:convDR}.
\end{proof}


\section{Examples}
\subsection{Numerical Demonstration}
We demonstrate the above results on the following synthetic numerical example. 
We construct a sparse object with $328$ uniform random positive and negative point-like sources in a 256-by-256 pixel field 
and randomly sample the Fourier transform of this object at a ratio of 1-to-8.  This yields $8192$ affine constraints. 
Local convergence results are illustrated in Figure \ref{f:SparseFourier} where the initial points $x^0$
are selected by uniform random $\left(-\delta/512,\delta/512\right)$ perturbations of the true solution 
in order to satisfy the assumptions of Theorems \ref{t:AP convergence} and \ref{t:convergence DR}.
The alternating projections  and Douglas-Rachford algorithms are shown respectively in panels (a)-(b) and (c)-(d) of Figure \ref{f:SparseFourier}.  
We show both the step lengths per iteration as well as the {\em gap distance} at each iteration defined as 
\begin{equation}\label{eq:gap}
   (\mbox{gap distance})^k\equiv \|P_{\sparse} x^k - P_Bx^k\|.
\end{equation}
Monitoring the gap allows one to ensure that the algorithm is indeed converging to a point 
of intersection instead of just a best approximation point.  
In panels (a) and (c) we set the sparsity parameter $s=328$, exactly the number of nonzero elements in the original 
image.  Panels (b) and (d) demonstrate the effect of overestimating the sparsity parameter,  $s=350$, on algorithm performance.  
The convergence of Douglas-Rachford for the case $s=350$ is not covered in our theory, however our numerical experiments indicate 
that one still achieves a linear-looking convergence over cycles, albeit with a very poor rate constant.   This remains to be 
proven.  
\begin{figure}[H]
(a) \includegraphics[width=0.45\linewidth]{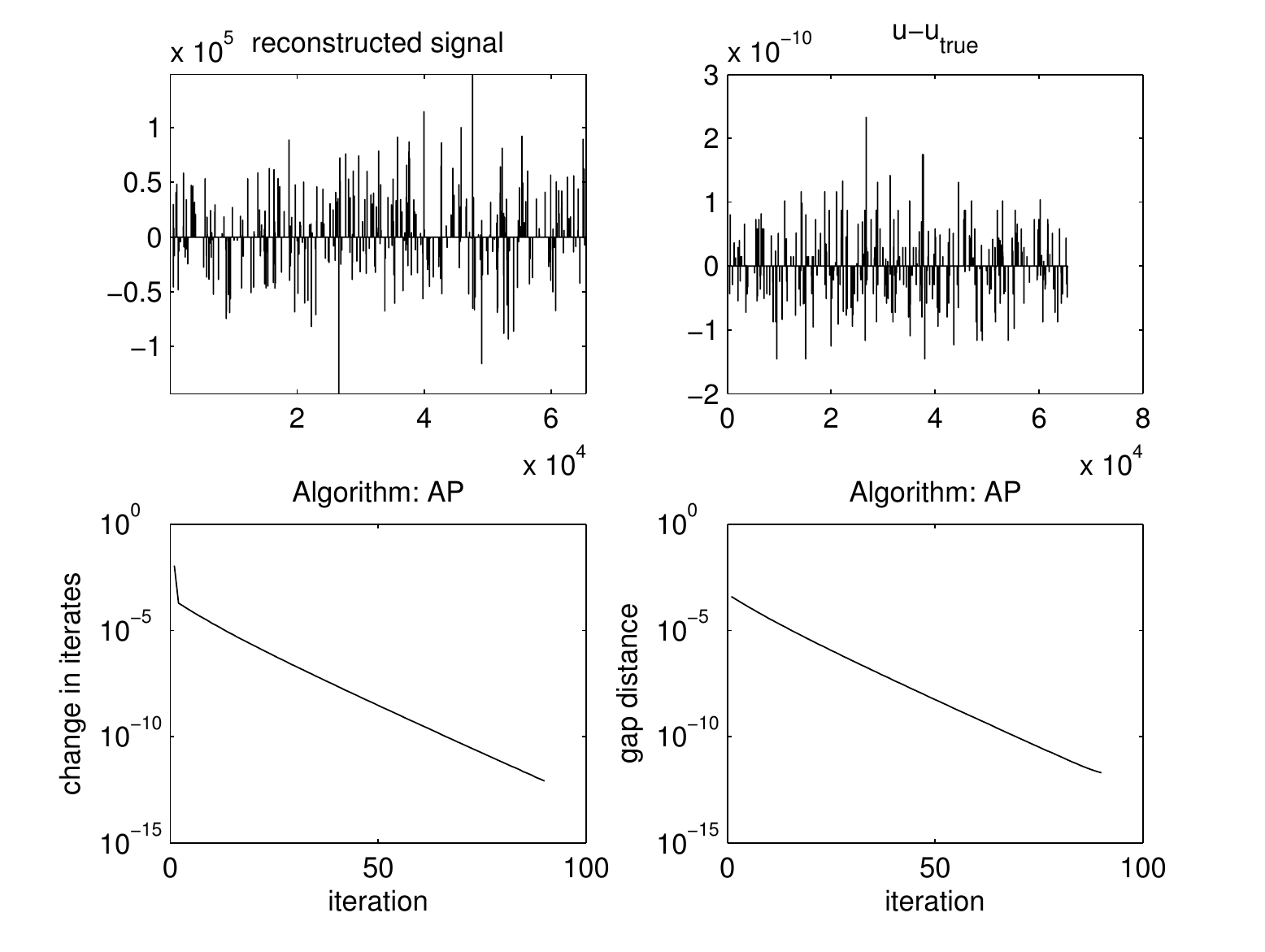}\hfill 
(b) \includegraphics[width=0.45\linewidth]{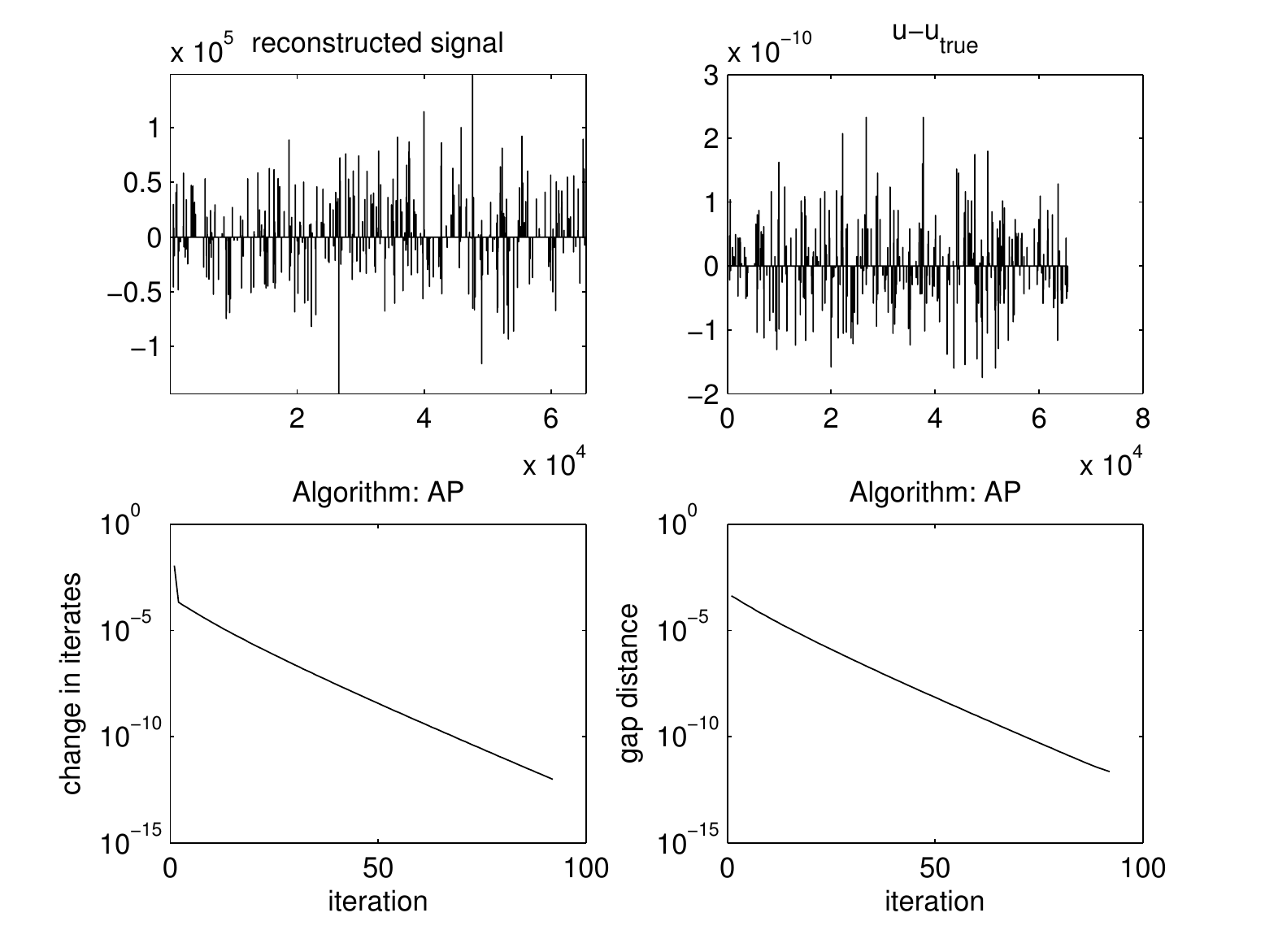}
(c) \includegraphics[width=0.45\linewidth]{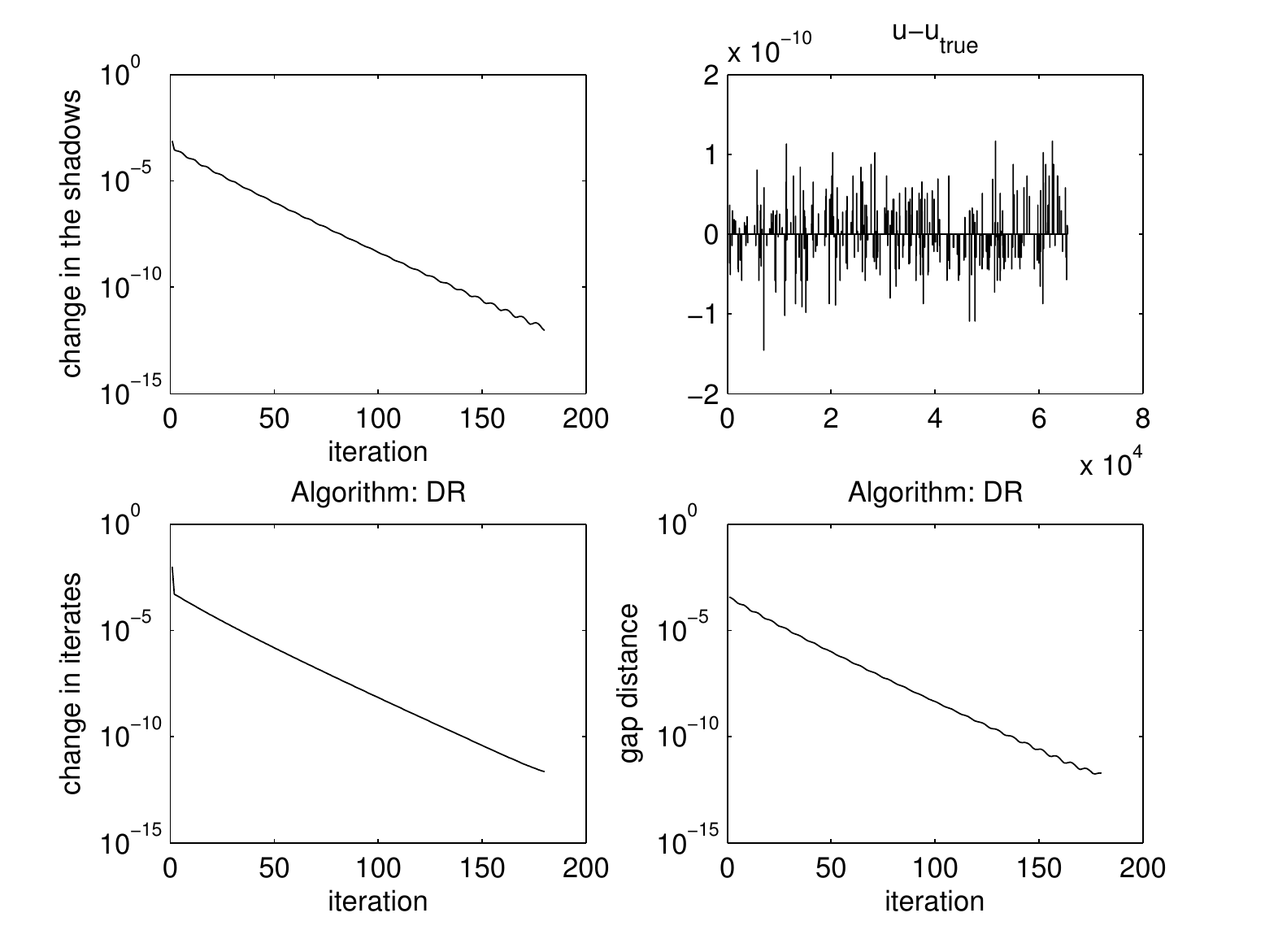}\hfill 
(d) \includegraphics[width=0.45\linewidth]{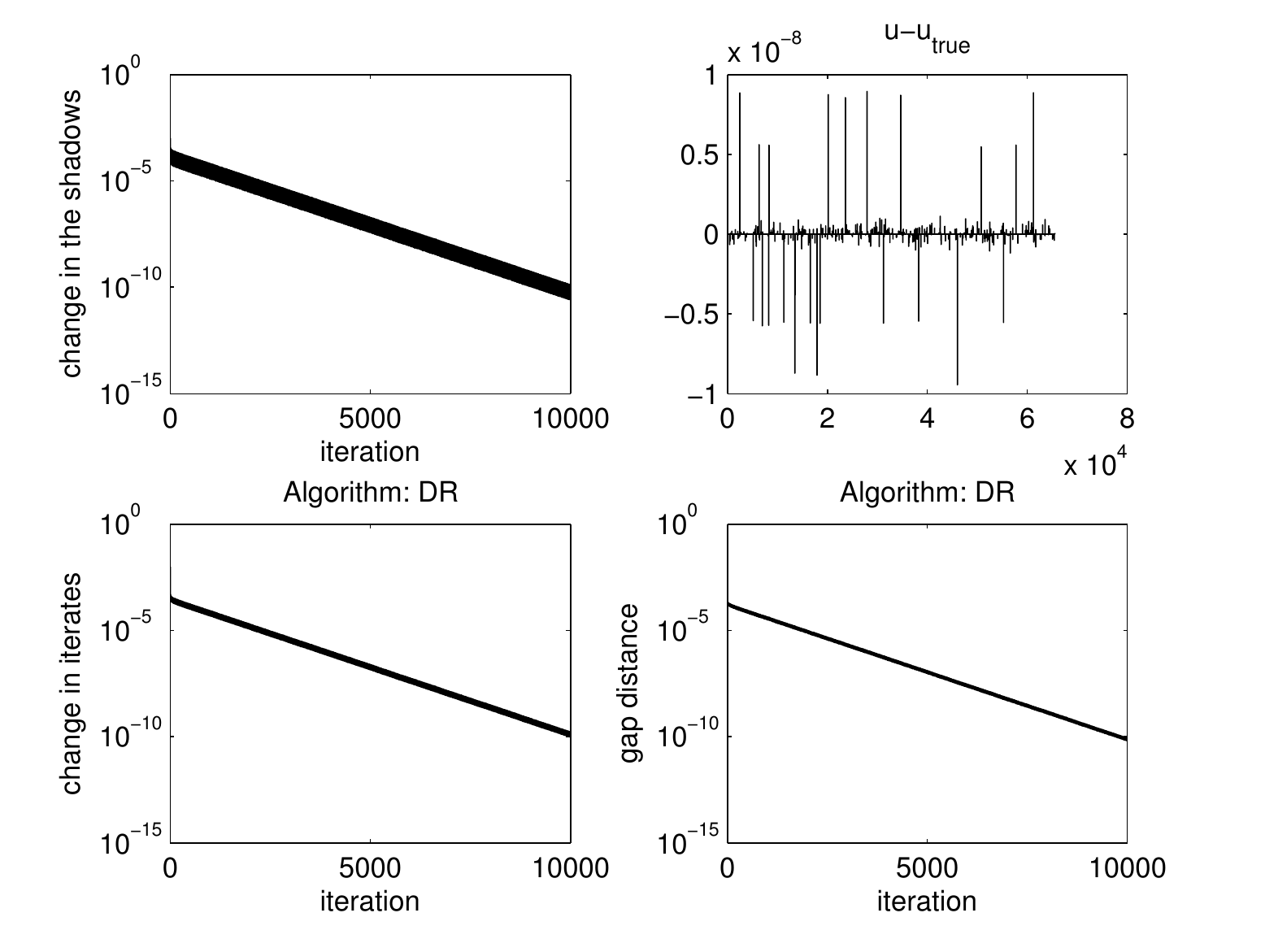}
\caption{ \label{f:SparseFourier}
(a) shows the convergence of alternating projections  in the case where the sparsity is exact, $s=328$. (b) shows the same with sparsity assumed too big, $s=350$.
In (c) and (d) we have the corresponding plots for Douglas-Rachford.  Case (d) is not covered by our theory. }
\end{figure}

The second synthetic example, shown in Figure \ref{f:RIP}, demonstrates global performance
of the algorithms and illustrates the results in Theorem
 \ref{t:AP convergence}, Theorem \ref{t:convergence DR} and Corollary \ref{t:API}. 
The solution is the vector $\point \equiv (10,0,0,0,0,0,0,0)$ and the affine subspace is the one generated 
by the matrix in \eqref{eq:matrix}. 
This matrix fulfills the assumptions of Corollary \ref{t:API}, as shown in Section \ref{ex:Ninja}. 
For the cases (a) and (c) the initial point $x^0$ can be written as $x^0\equiv \point + u$ where $u$ is a vector with 
uniform random values from the interval $(-1,1)$. 
The initial values hence fulfill the assumptions of Theorems \ref{t:AP convergence} and \ref{t:convergence DR}.
For (b) and (d) again the initial point $x^0$ can be written as $x^0\equiv \point + u$ while $u$ is now a 
vector with uniform random values from the interval $(-100,100)$. 
As expected, the sequence of alternating projections  converges to the true solution in (c).
The case for Douglas-Rachford however, shown in (d), is not covered by our theory.

\begin{figure}[H]
(a)\includegraphics[scale=.45]{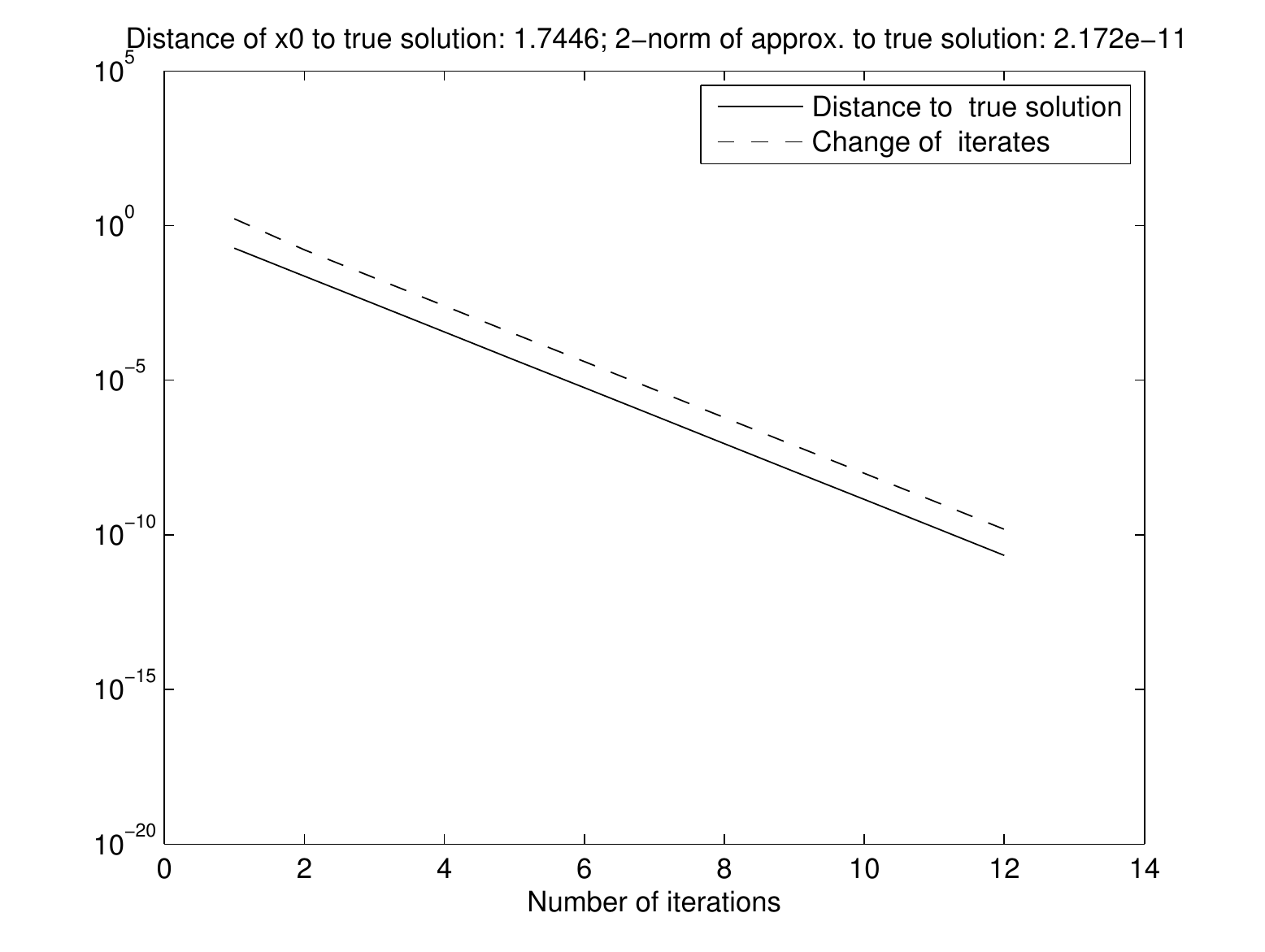}\hfill
(b)\includegraphics[scale=.45]{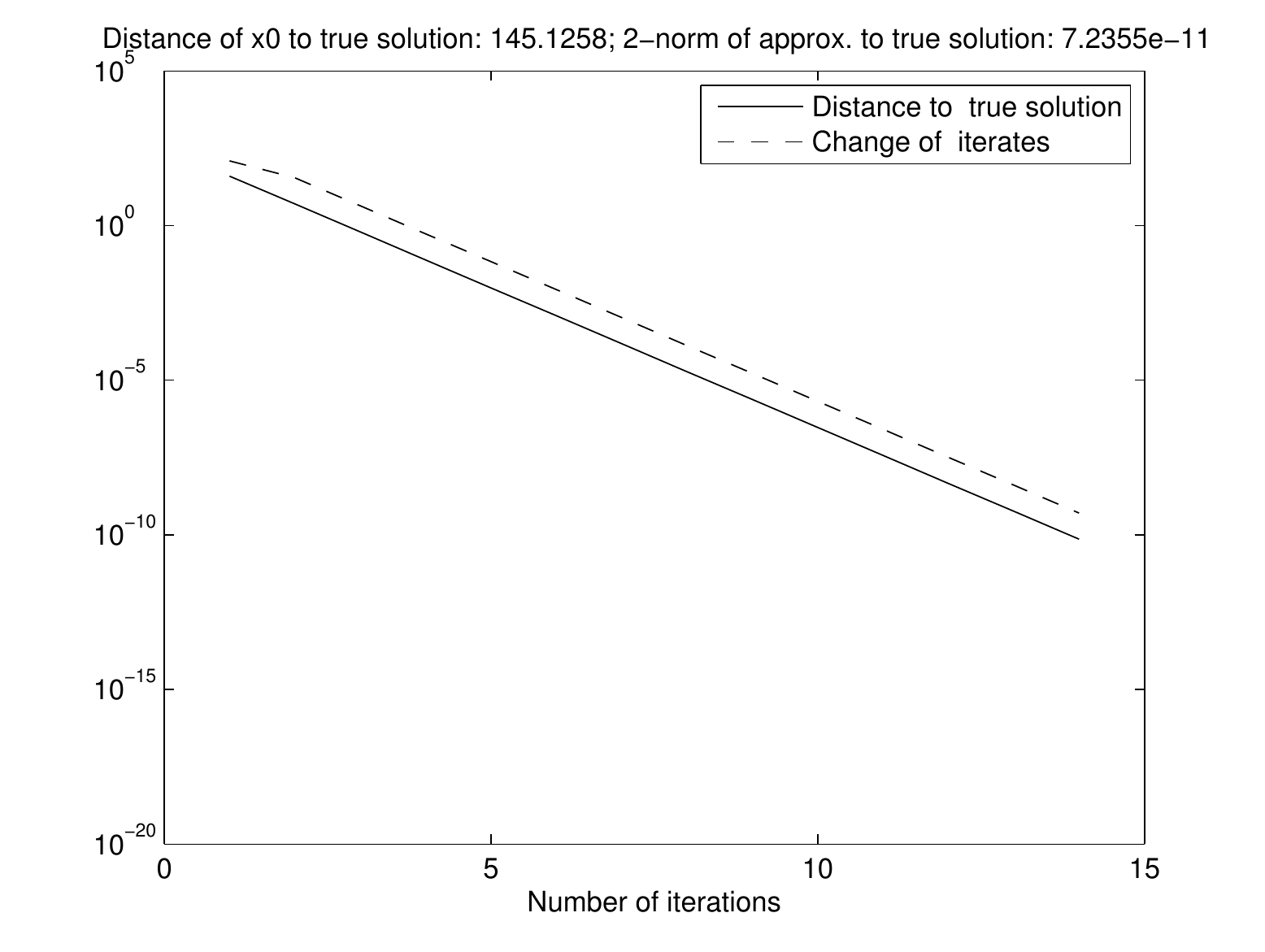}
(c)\includegraphics[scale=.45]{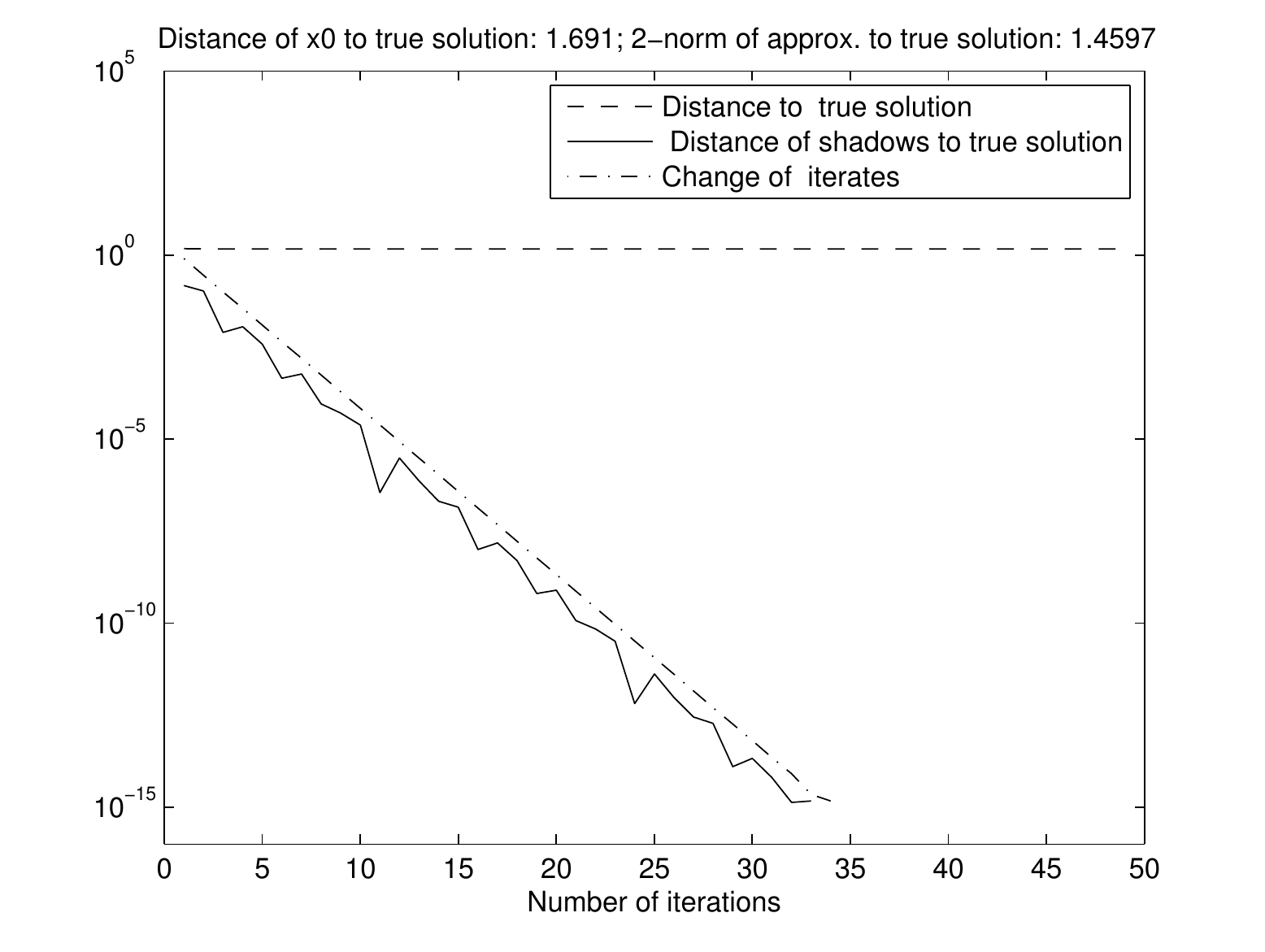}\hfill
(d)\includegraphics[scale=.45]{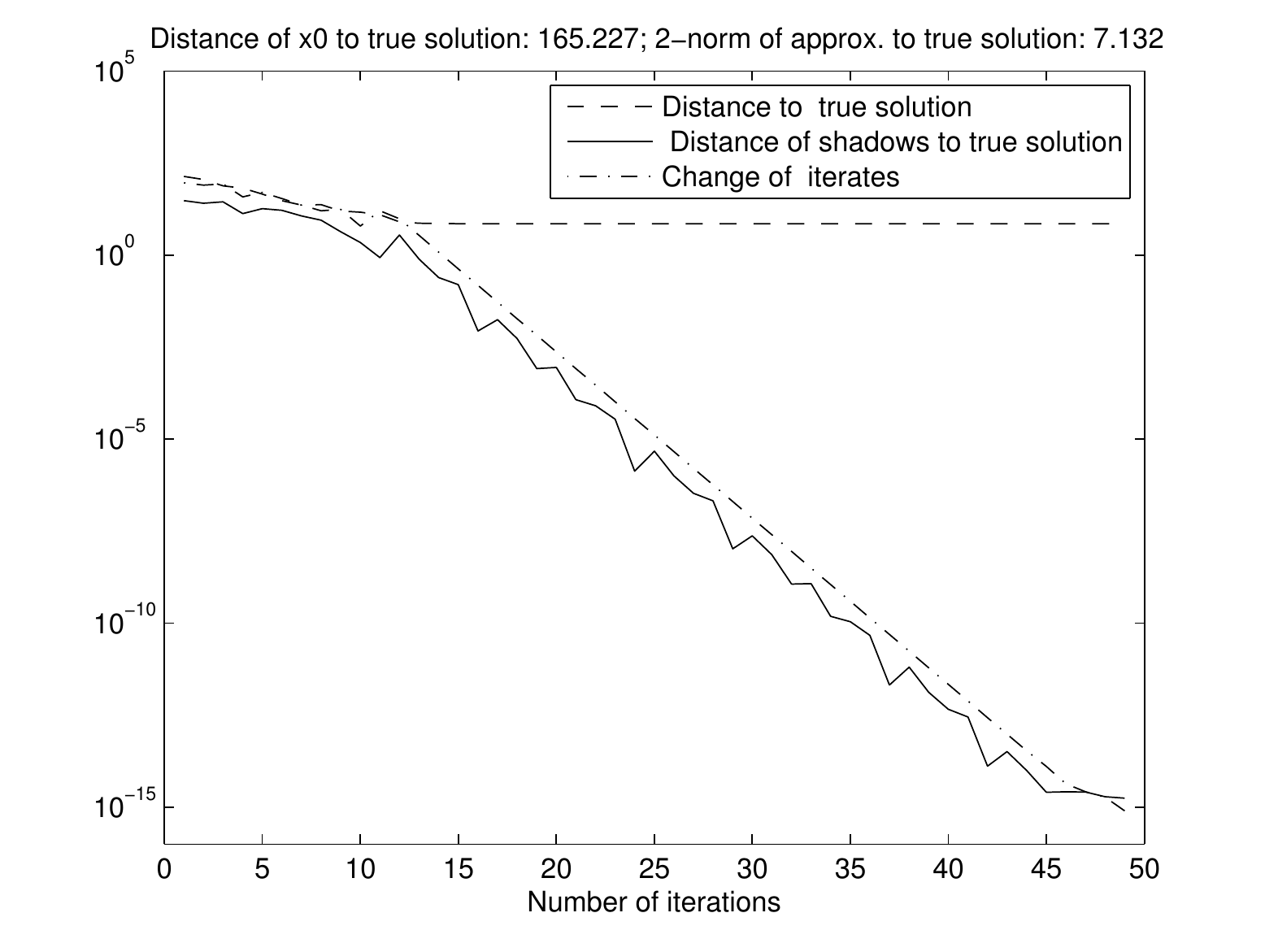}
\caption{\label{f:RIP} Example with an affine subspace generated by the matrix from Section \ref{ex:Ninja}: 
(a) shows the local convergence as shown in Theorem \ref{t:AP convergence}, 
(b) is an example of global convergence of alternating projections  as stated in Corollary \ref{t:API}. 
(c) is an example of local convergence of Douglas-Rachford to its fixed point set while the shadows converge to the intersection,  
as proven in Theorem \ref{t:convergence DR}. This example also shows that the iterates converge to a fixed point that is not in 
the intersection, as proven in Theorem \ref{t:AAR convergence}.
The plot (d) is an example where Douglas-Rachford appears to converge globally. 
This behavior is not covered by our theory. }
\end{figure}

\subsection{Analytic examples}
\subsubsection{Example of a matrix satisfying assumptions of Corollary \ref{t:API}}\label{ex:Ninja}
Finding nonsquare matrices satisfying \eqref{eq:RIP} or deciding whether or not a matrix fulfills some similar
condition is, in general, hard to do -- but not impossible.  
In this section we provide a concrete example of a nonsquare matrix satisfying the assumptions of 
Corollary \ref{t:API}. 

We take the matrix
\begin{equation}\label{eq:matrix}
M =  \frac{1}{\sqrt{8}}\begin{pmatrix}
1 & 1 & 1 & 1 & 1 & 1 & 1 & 1 \\ 1 & 1 & 1 & 1 & -1 & -1 & -1 & -1 \\ 1 & 1 & -1 & -1 & 1 & 1 & -1 & -1 \\ 1 & -1 & 1 & -1 & 1 & -1 & 1 & -1 \\ 1 & 1 & -1 & -1 & -1 & -1 & 1 & 1 \\ 1 & -1 & -1 & 1 & 1 & -1 & -1 & 1 \\ 1 & -1 & 1 & -1 & -1 & 1 & -1 & 1 
\end{pmatrix}   
\end{equation}
The rows of $M$ are pairwise orthogonal and so $MM^\top =\Id _7$.
We compute the constant $\delta$ in \eqref{eq:RIP} for $s = 2$ to get a result for the recovery of 
1-sparse vectors with alternating projections.  Recall that $s$ can be  
larger than the sparsest feasible solution (see  Remark \ref{r:s too big}). 
In general, a normalized 2-sparse vector in $\R^8$ has the form
\[x = (\cos(\alpha), \sin(\alpha),0,0,0,0,0,0),\]
where the position of the $\sin$ and of the $\cos$ are of course arbitrary.
The squared norm of the product $Mx$ is equal to
\[ \|Mx\|_2^2 = \frac{1}{8}\sum_{i=1}^7|\cos(\alpha) + z_i\sin(\alpha)|^2 ,  \]
where $z_i\in\{-1,1\}$. We note that the inner products of distinct columns of $M$ are $-1,1$, so $\sum_{i=1}^7 z_i = \pm 1$. Then
\begin{align*}
& \frac{1}{8}\sum_{i=1}^7|\cos(\alpha) + z_i\sin(\alpha)|^2 \\
= & \frac{1}{8}\sum_{i=1}^7\cos(\alpha)^2 + 2z_i\sin(\alpha)\cos(\alpha) + \sin(\alpha)^2\\
= & \frac{1}{8}(7\pm\sin(2\alpha))\in\left[\frac{3}{4}, 1\right].
\end{align*}
This means that $\frac{3}{4}\enorm{x}^2\leq \enorm{Mx}^2 \leq \enorm{x}^2 \quad \forall x\in A_2$, 
where $A_2$ is the set of 2-sparse vectors in $\R^8$.
In other words, we can recover any $1$-sparse vector with the method of alternating projections.

\subsubsection{An easy example where alternating projections and Douglas-Rachford iterates don't converge}
The following example, discovered with help from Matlab's Symbolic Toolbox, shows some of the more interesting 
pathologies that one can see with these algorithms when not starting sufficiently close to a solution.

Let  $n=3, m=2, s=1$ and 
\[
M = \begin{pmatrix}1 & -\frac{1}{2} & 0 \\ 0 & \frac{1}{2} & -1\end{pmatrix},\qquad 
p=\begin{pmatrix}-5\\5\end{pmatrix}\]
The point $(0,10,0)^\top$ is the sparsest solution to the equation $Mx = p$ and the affine space $\B$ is 
\[
\B =\begin{pmatrix}0  \\ 10 \\ 0 \end{pmatrix} + \lambda \begin{pmatrix}1  \\ 2 \\ 1 \end{pmatrix},\qquad \quad \textup{with}\ \lambda\in \R.
\]
If we take the initial point (apologies for the numbers!)
$$x^0 = \left( \tfrac{38894857328700073}{237684487542793012780631851008},\tfrac{-297105609428507214758454580565}{118842243771396506390315925504}, 
\tfrac{-1188422437713940163629828887893}{237684487542793012780631851008} \right),$$
then $T_{DR}x^0 = x^0 + (-5,0,5)$ and $T_{DR}^2x^0 = x^0$.

Note that this example is different from the case in Theorem \ref{t:AAR convergence}: in Theorem \ref{t:AAR convergence} 
we establish that,  if $s< \textup{rank}(M)$, then the fix point set of $T_{DR}$ is strictly larger than the solution set to problem 
\eqref{eq:feasibility}. The concrete case detailed here also satisfies $s< \textup{rank}(M)$, however, with the given $x^0$ we are 
not near the set of fixed points, but in a {\em cycle} of $T_{DR}$.

If, on the other hand,  we take the point $\hat{x}^0 = \left( -4 , 0, 0 \right)$, then $P_B\hat{x}^0 = \left( -4 , 2, -4 \right)$
and the set $P_{A_1}P_{B}\hat{x}^0$ is equal to $\left\lbrace \left( -4 , 0, 0 \right), \left( 0 , 0, -4 \right) \right\rbrace$.
The projection $P_B\left( 0 , 0, -4 \right)$ is again the point $\left( -4 , 2, -4 \right)$.
This shows that the alternating projection \eqref{eq:AP} iteration is stuck at the points $\left( -4 , 0, 0 \right)$ and 
$\left( 0 , 0, -4 \right)$ which are clearly not in the intersection $A_{1}\cap B=\{(0,10,0)^\top \}$.  This also highlights 
a manifestation of the multivaluedness of the projector $P_{A_1}$.

\section{Conclusion}
The usual avoidance of nonconvex optimization over convex relaxations is not always warranted.  
In this work we have determined sufficient conditions under which simple algorithms applied to 
nonconvex sparse affine feasibility are guaranteed to converge globally at a linear rate.   We have also shown 
local convergence of the prominent Douglas-Rachford algorithm applied to this problem.  These results are 
intended to demonstrate the potential of recently developed analytical tools for understanding nonconvex 
fixed-point algorithms in addition to making the broader point that nonconvexity is not {\em categorically} bad. 
That said, the global results reported here rely heavily on the linear structure of the problem, and 
local results are of limited practical use.  
Of course, the decision about whether to favor a convex relaxation over a nonconvex formulation 
depends on the structure of the problem at hand and many open questions remain. 
First and foremost among these is:  what are {\em necessary} conditions for global convergence of 
simple algorithms to global solutions of nonconvex problems? The apparent robust global behavior of 
Douglas-Rachford has eluded explanation.  What are conditions for global convergence of the 
Douglas-Rachford algorithm for this problem?  What happens to these algorithms when the chosen 
sparsity parameter $s$ is too small?   At the heart of these questions lies a long-term research program into 
regularity of nonconvex variational problems, the potential impact of which is as broad as it is deep. 

\section*{Acknowledgements}  We thank the anonymous referees for their thorough and helpful suggestions for 
improving the original manuscript.  

\addcontentsline{toc}{section}{Literature}

\end{document}